\documentclass[amstex,12pt]{article}
\usepackage{amsfonts,cite,color}
\usepackage{amssymb,amsmath,amsthm,mathrsfs}
\newtheorem{theorem}{Theorem}[section]

\newtheorem{lemma}{Lemma}[section]
\newtheorem{example}{Example}[section]
\newtheorem{definition}{Definition}[section]
\newtheorem{remark}{Remark}[section]
\newtheorem{proposition}{Proposition}[section]

\newcommand{\beq}{\begin{equation}}
\newcommand{\eeq}{\end{equation}}
\newcommand{\beqn}{\begin{eqnarray}}
\newcommand{\eeqn}{\end{eqnarray}}

\def\to{{\rightarrow}}

\begin{document}

\allowdisplaybreaks
\title{Weyl double-measure pseudo almost automorphic functions  and  Weyl double-measure pseudo almost automorphic solutions to a semilinear abstract differential equations\thanks{This work is supported by the National Natural Science Foundation of China
under Grant No. 12261098.}
}
\author{Yongkun Li\thanks{ Email: yklie@ynu.edu.cn.}\\
 $^b$Department of Mathematics, Yunnan University,\\
 Kunming, Yunnan 650091,\\
 People's Republic of China}
\date{}
\maketitle{}
\begin{abstract}
This paper first propose a concept of Weyl double-measure pseudo-almost automorphic functions and examines their fundamental characteristics. Subsequently, employing fixed point theorems, we systematically investigate the existence and uniqueness of both Weyl almost automorphic solutions and Weyl double-measure pseudo-almost automorphic solutions for a class of semilinear abstract differential equations. Finally, through the application of inequality-based analytical methods, we establish the global exponential stability of these solutions.
\end{abstract}


\section{ Introduction }
\setcounter{equation}{0}
 \indent

Almost automorphic functions, as a natural generalization of Bohr almost periodic functions\cite{bo}, were introduced by Bocher in 1955 during his investigation of a differential geometry problem \cite{bc}. Since then, the existence problem of almost automorphic solutions for differential equations and dynamical systems has become one of the central topics in these research fields. Concurrently, various extensions of almost automorphic functions have been continuously proposed, including Stepanov almost automorphic functions \cite{zs1}, Weyl almost automorphic functions \cite{weyl1}, and Besicovitch almost automorphic functions \cite{kos1,kos2}. Following Zhang's groundbreaking introduction of Bohr-type pseudo almost periodic functions \cite{zhang1,zhang2}, mathematical research has witnessed an influx of pseudo almost periodic and pseudo almost automorphic function concepts under different frameworks. Notably, the study of pseudo almost automorphy in differential equations \cite{xiao} and their applications to realistic mathematical models has emerged as a significant research focus \cite{ddd1,w1,w2,w3,w4,w5,w6}. In the realm of pseudo almost automorphic functions, numerous variants have been developed, such as weighted pseudo almost automorphic functions \cite{jq}, single-measure pseudo almost automorphic functions \cite{dd1}, double-measure pseudo almost periodic functions \cite{dd2,dd3}, Stepanov pseudo almost automorphic functions \cite{dia}, Stepanov weighted pseudo almost automorphic functions \cite{xia}, Stepanov measure pseudo almost automorphic functions \cite{uk} and Weyl pseudo almost automorphic functions \cite{abb}. It is particularly noteworthy that among these pseudo almost automorphic concepts, the double-measure pseudo almost automorphy represents the most comprehensive framework proposed to date. However, there currently exist no published articles addressing Weyl weighted pseudo almost automorphic functions, Weyl measure pseudo almost automorphic functions, and the corresponding Weyl pseudo almost automorphic solutions of differential equations.

Motivated by these observations, this work primarily aims to put forward a concept of Weyl double-measure pseudo almost automorphic functions and systematically investigate their fundamental properties. Subsequently, as an application of our theoretical framework, we will examine the existence and global exponential stability of Weyl double-measure pseudo almost automorphic solutions for a class of semilinear abstract differential equations. This investigation bridges current theoretical gaps while providing new analytical tools for studying Weyl measure pseudo almost automorphic dynamics of differential equations.

The organization of this paper is as follows: Section 2 recalls essential concepts, introduces necessary notations, and presents the definition and fundamental properties of Weyl almost automorphic functions. We define the working spaces required for subsequent analysis and provide a rigorous proof of their completeness. Section 3 proposes a concept of Weyl double-measure pseudo almost automorphic functions and systematically investigates their core characteristics. Section 4 establishes the existence and uniqueness of Weyl pseudo almost automorphic solutions and Weyl double-measure pseudo almost automorphic solutions for a class of semilinear abstract differential equations through fixed-point theorem applications. Section 5 conducts an analysis of the global exponential stability of these solutions. Section 6 provides an example to validate the theoretical framework developed in this study.

\section{Preliminaries and some properties of Weyl almost automorphic functions}
\setcounter{equation}{0}
\indent

Let $(\mathbb{X},\|\cdot\|)$ and $(\mathbb{Y},\|\cdot\|_\mathbb{Y})$ be Banach spaces.
 For exponents $1\le p <\infty$, we write $\mathcal{L}_{loc}^{p}(\mathbb{R},\mathbb{X})$ for the collection  of all mappings $f:\mathbb{R} \rightarrow\mathbb{X}$ whose restrictions to any compact interval are $p$-integrable.  The symbol $\mathcal{L}^\infty(\mathbb{R},\mathbb{X})$ comprises all measurable functions  $g:\mathbb{R} \rightarrow\mathbb{X}$ satisfying $\|g\|_\infty:=\text{ess} \sup_{t\in \mathbb{R}}\|g(t)\|$,  endowed with the essential supremum norm $\|\cdot\|_\infty$.

\begin{definition}\label{bo1} \cite{a4} A continuous mapping $\varphi: \mathbb{R}\to\mathbb{X}$ is termed Bohr almost periodic if it satisfies the following uniform approximation criterion: for every $\epsilon>0$, there corresponds a positive number $\ell_\epsilon$ such that every real interval of length $\ell_\epsilon$ contains at least one $\tau=\tau_\epsilon$ (called an $\epsilon$-translation number) fulfilling
\begin{align*}
\sup_{t\in\mathbb{R}} \|\varphi(t+\tau) - \varphi(t)\| < \epsilon.
\end{align*}
The space constituted by all such Bohr almost periodic functions will be symbolized as $AP(\mathbb{R},\mathbb{X})$.
\end{definition}

\begin{definition}\label{bo2}\cite{ddd1} A continuous mapping $\varphi:\mathbb{R}\to\mathbb{X}$ is declared Bochner almost automorphic if it satisfies the double limit criterion: given any real sequence $\{\gamma'_n\}\subset \mathbb{R}$, there exists a refined subsequence $\{\gamma_n\}\subset\{\gamma'_n\}$ such that the two-stage limit process
\begin{align*}
\lim_{n\to\infty} \varphi(t+\gamma_n) := \widetilde{\varphi}(t) \quad \text{(pointwise convergence)}
\end{align*}
establishes a well-defined function $\widetilde{\varphi}:\mathbb{R}\to\mathbb{X}$, and furthermore satisfies the reciprocal convergence
\begin{align*}
\lim_{n\to\infty} \widetilde{\varphi}(t-\gamma_n) = \varphi(t) \quad \text{(pointwise restoration)}
\end{align*}
for all $t\in\mathbb{R}$. The collection of all such mappings forms the Bochner almost automorphic function space, denoted by $AA(\mathbb{R},\mathbb{X})$.
\end{definition}

\begin{definition}\cite{bb} A continuous mapping $\varphi:\mathbb{R}\times\mathbb{R}\to\mathbb{Y}$ is defined as bi-almost automorphic if it satisfies the following dual convergence mechanism: for any real sequence $\{\gamma'_n\}\subset\mathbb{R}$, one can extract a refined subsequence $\{\gamma_n\}\subset\{\gamma'_n\}$ such that the coupled limits
\begin{align*}
\lim_{n\to\infty} \varphi(t+\gamma_n, s+\gamma_n) := \widetilde{\varphi}(t,s) \quad \text{(simultaneous convergence)}
\end{align*}
determine a well-defined function $\widetilde{\varphi}:\mathbb{R}^2\to\mathbb{Y}$ for each pair $(t,s)\in\mathbb{R}^2$, and reciprocally satisfy the restoration condition
\begin{align*}
\lim_{n\to\infty} \widetilde{\varphi}(t-\gamma_n, s-\gamma_n) = \varphi(t,s) \quad \text{(symmetric restoration)}
\end{align*}
for every pair $(t,s)\in\mathbb{R}^2$.
\end{definition}

For $u\in \mathcal{L}_{loc}^{p}(\mathbb{R},\mathbb{X})$ and $h\in \mathbb{R}^+$, the Stepanov norm of $u$ is defined as follows:
\begin{align*}
\|u\|_{S_h^p}=\sup\limits_{t\in\mathbb{R}}\bigg(\frac{1}{h}\int_{t}^{t+h}\|u(s)\|^p\mathrm{d}s\bigg)^{\frac{1}{p}}.
\end{align*}
From the definition $\|\cdot\|_{S_h^p}$, it seems that for different $h\in \mathbb{R}^+$, we have different Stepanov norm. However, for all $h\in \mathbb{R}$, all the Stepanov norms are equivalent. Hence, for convenience, we have the following definition:
\begin{definition}\label{def24}\cite{zs1} A function $u\in\mathcal{L}^p_{loc}(\mathbb{R},\mathbb{X})$ is characterized as Stepanov almost automorphic if it satisfies that for given any real sequence $\{\gamma'_n\}\subset\mathbb{R}$, one can extract a subsequence $\{\gamma_n\}\subset\{\gamma'_n\}$ and identify a function $\widetilde{u}\in\mathcal{L}^p_{loc}(\mathbb{R},\mathbb{X})$ such that
\begin{align*}
\lim_{n\to\infty} \left(\int_t^{t+1} \|u(s+\gamma_n) - \widetilde{u}(s)\|^p ds\right)^{1/p} = 0
\end{align*}
holds for each $t\in\mathbb{R}$, and reciprocally satisfies
\begin{align*}
\lim_{n\to\infty} \left(\int_t^{t+1} \|\widetilde{u}(s-\gamma_n) - u(s)\|^p ds\right)^{1/p} = 0
\end{align*}
for every $t\in\mathbb{R}$. The  space constituted by all Stepanov-type almost automorphic functions is symbolized by $S^p_{AA}(\mathbb{R},\mathbb{X})$.
\end{definition}

 For $u\in \mathcal{L}_{loc}^{p}(\mathbb{R},\mathbb{X})$, the Weyl  seminorm of $u$ is defined as follows:
\begin{align*}
\|u\|_{W^p}=\lim\limits_{h\rightarrow\infty}\|u\|_{S_h^p}=\lim\limits_{h\rightarrow\infty}\sup\limits_{\theta\in\mathbb{R}}\bigg(\frac{1}{h}\int_{\theta}^{\theta+h}\|u(s)\|^p\mathrm{d}s\bigg)^{\frac{1}{p}}.
\end{align*}

\begin{definition}\label{def11}
 A function $u\in \mathcal{L}_{loc}^p(\mathbb{R},\mathbb{X})$ is called $W^p$-bounded if  $||u||_{W^p}<\infty.$
We denote by  $W^p_{B}(\mathbb{R},\mathbb{X})$ the family of all $W^p$-bounded functions.
\end{definition}

From Definition \ref{def11}, one can easily show the following two lemmas.

\begin{lemma}\label{lem21}
If $u\in \mathcal{L}_{loc}^{p}(\mathbb{R},\mathbb{X})$, then for every $\sigma\in \mathbb{R}$, it holds that
\begin{align*}
\|u(\cdot)\|_{W^p}=\|u(\cdot+\sigma)\|_{W^p}.
\end{align*}
\end{lemma}

\begin{lemma}\label{lem11}
If  $u,v\in W^p_{B}(\mathbb{R},\mathbb{X})$ and $\alpha$ is a scalar,  then  $u+v, \alpha u\in W^p_{B}(\mathbb{R},\mathbb{X})$.
\end{lemma}

\begin{theorem}\label{lem12}
The space $W^p_{B}(\mathbb{R},\mathbb{X})$ is a complete space regarding the seminorm $\|\cdot\|_{W^p}$.
\end{theorem}
\begin{proof}
Let $\{\varphi_n\}\in W^p_{B}(\mathbb{R},\mathbb{X})$  be a Cauchy sequence of $W^p_{B}(\mathbb{R},\mathbb{X})$. Then for any $\epsilon>0$, there is a positive integer $N(\epsilon)$ such that for any $n,m\geq N$, it holds
\begin{align*}
 \|\varphi_n-\varphi_m\|_{W^p}= \lim\limits_{h\rightarrow\infty}\sup\limits_{\theta\in\mathbb{R}}\bigg(\frac{1}{h}\int_{\theta}^{\theta+h}\|\varphi_n(s)-\varphi_m(s)\|^p\mathrm{d}s\bigg)^{\frac{1}{p}}<\epsilon,
\end{align*}
which implies that there exists a number $h_0(\epsilon)>0$ such that for all $h\geq h_0$,
\begin{align*}
  \sup\limits_{\theta\in\mathbb{R}}\bigg(\frac{1}{h}\int_{\theta}^{\theta+h}\|\varphi_n(s)-\varphi_m(s)\|^p\mathrm{d}s\bigg)^{\frac{1}{p}}<2\epsilon.
\end{align*}
In the above inequality, let $h(\geq h_0)$ be fixed, then  we have, for all $m,n\geq N$,
\begin{align*}
\int_{\theta}^{\theta+h}\|\varphi_n(s)-\varphi_m(s)\|^p\mathrm{d}s<h (2\epsilon)^p,\,\, \forall \, \theta\in \mathbb{R}.
\end{align*}
This inequality indicates that for every $\theta\in \mathbb{R}$,  $\{\varphi_n(t)\}$ is a Cauchy sequence of $\mathcal{L}^p([\theta,\theta+h],\mathbb{X})$ when it is restricted on $[\theta,\theta+h_0]$. Since $\mathcal{L}^p([\theta,\theta+h],\mathbb{X})$ is a Banach space, there is a $\varphi\in \mathcal{L}^p([\theta,\theta+h],\mathbb{X})$ such that $\lim\limits_{n\rightarrow \infty}\varphi_n(t)=\varphi(t)$.
Since the limit is unique (a.e.) in any space $\mathcal{L}^p([\theta,\theta+h],\mathbb{X})$, where $h\geq h_0$ and $\theta\in \mathbb{R}$,  we obtain that there is a  $\varphi\in \mathcal{L}_{loc}^p(\mathbb{R},\mathbb{X})$ satisfying $\lim\limits_{n\rightarrow \infty}\varphi_n(t)=\varphi(t)$ in $\mathcal{L}^p_{loc}(\mathbb{R},\mathbb{X})$.

For the $\varphi$ above, for each $h\geq h_0$ and each $n\geq N$, there holds
\begin{align*}
 \sup\limits_{\theta\in\mathbb{R}} \bigg(\frac{1}{h}\int_{\theta}^{\theta+h}\|\varphi_n(s)-\varphi(s)\|^p\mathrm{d}s\bigg)^{\frac{1}{p}}\leq 2\epsilon.
\end{align*}
Hence, for $n\geq N$,
\begin{align*}
  \lim\limits_{h\rightarrow\infty} \sup\limits_{\theta\in\mathbb{R}}\bigg(\frac{1}{h}\int_{\theta}^{\theta+h}\|\varphi_n(s)-\varphi(s)\|^p\mathrm{d}s\bigg)^{\frac{1}{p}}\leq 2\epsilon.
\end{align*}

Since $\varphi-\varphi_n, \varphi_n$ are $W^p$-bounded and $\varphi=(\varphi-\varphi_n)+\varphi_n$, we conclude that $\varphi$ is $W^p$-bounded.
 The proof is complete.
\end{proof}

Now, we introduce the definition of Weyl almost automorphic functions as follows.
\begin{definition}\label{zs} \cite{lh2}
A function $\varphi\in W^p_B(\mathbb{R},\mathbb{X})$ is said to be $p$-th Weyl almost automorphic, if for every sequence  $\{\gamma'_{n}\}\subset \mathbb{R}$, there exist a subsequence $\{\gamma_{n}\}\subset \{\gamma'_{n}\}$ and a function $\tilde{\varphi}: \mathbb{R}\rightarrow \mathbb{X}$ such that
$$\lim_{n\rightarrow\infty}\|\varphi(\cdot+\gamma_{n})-\tilde{\varphi}(\cdot)\|_{W^{p}}=0\quad and \quad \lim_{n\rightarrow\infty}\|\tilde{\varphi}(\cdot-\gamma_{n})-\varphi(\cdot)\|_{W^{p}}=0.$$
 We denote by $W^p_{AA}(\mathbb{R},\mathbb{X})$ the collection of all such functions.
\end{definition}
From Definitions \ref{bo1}, \ref{bo2}, \ref{def24} and \ref{zs}, it is easy to see that
$$AP(\mathbb{R},\mathbb{X})\subset AA(\mathbb{R},\mathbb{X}) \subset S^p_{AA}(\mathbb{R},\mathbb{X})\subset W^p_{AA}(\mathbb{R},\mathbb{X}).
$$

\begin{example}\label{ex1a}
According to Example 4.4 in \cite{a4}, we see that
\begin{equation*}
 f(t):=\cos \left(\frac{1}{2+\sin \pi t+\sin 3 t}\right) \in A A(\mathbb{R}, \mathbb{R}).
\end{equation*}
Since $\|\alpha\|_{W^p}=0$ and $\|\beta\|_{W^p}=0$, where $\alpha(t)=3e^{-|t|}$ and $\beta(t)=\frac{2}{1+t^2}$, we have
\begin{equation*}
 f(t)+\alpha(t)+\beta(t) \in W_{AA}^p(\mathbb{R}, \mathbb{R}).
\end{equation*}
\end{example}

\begin{example}\label{ex1b}
It is easy to check that function
\begin{align*}
  f(t)=\left\{\begin{array}{lll} 1, & 0<t<\frac{1}{2},\\
  0, & \mathrm{elsewhere}\end{array}\right.
\end{align*}
belongs to $W^p_{AA}(\mathbb{R},\mathbb{R})$ with $\tilde{f}=0$.
\end{example}

\begin{proposition}\label{lem23}
If $f\in W^{p}_{AA}(\mathbb{R},\mathbb{X})$, then $\tilde{f}\in W^p_B$ and $\|\tilde{f}\|_{W^p}=\|f\|_{W^p}$, where $\tilde{f}$ is mentioned in Definition \ref{zs}.
\end{proposition}
\begin{proof}
For any sequence of  $\{\gamma'_{n}\}\subset \mathbb{R}$, there exists a subsequence $\{\gamma_{n}\}$ such that
$$\lim_{n\rightarrow\infty}\|f(\cdot+\gamma_{n})-\tilde{f}(\cdot)\|_{W^{p}}=0\quad \mathrm{and} \quad \lim_{n\rightarrow\infty}\|\tilde{f}(\cdot-\gamma_{n})-f(\cdot)\|_{W^{p}}=0.$$
Hence, for any $\epsilon>0$, there exists an $N>0$ such that for $n>N$,
$$\|f(\cdot+\gamma_{n})-\tilde{f}(\cdot)\|_{W^{p}}<\epsilon\quad \mathrm{and} \quad \|\tilde{f}(\cdot-\gamma_{n})-f(\cdot)\|_{W^{p}}<\epsilon.$$
Therefore, for $n>N$, we have
\begin{align*}
  \|\tilde{f}(\cdot)\|_{W^p}\leq \|f(\cdot+\gamma_{n})-\tilde{f}(\cdot)\|_{W^{p}}+\|f(\cdot+\gamma_{n})\|_{W^{p}}<\epsilon+\|f(\cdot)\|_{W^{p}}
\end{align*}
and
\begin{align*}
 \|f(\cdot)\|_{W^p}\leq  \|\tilde{f}(\cdot-\gamma_{n})-f(\cdot)\|_{W^{p}}+\|\tilde{f}(\cdot-\gamma_{n})\|_{W^{p}}<\epsilon+\|\tilde{f}(\cdot)\|_{W^{p}}.
\end{align*}
From the arbitrariness of $\epsilon$, the conclusion of the proposition follows immediately. This ends the proof.
\end{proof}

\begin{lemma}\label{lem1}
If $u,v\in W^p_{AA}(\mathbb{R},\mathbb{X})$ and  $\alpha\in \mathbb{R},$ then we have $u+v,\alpha u\in W^p_{AA}(\mathbb{R},\mathbb{X}).$
\begin{proof}
Since $u,v\in W^p_{AA}(\mathbb{R},\mathbb{X}),$  for any sequence  $\{\gamma''_{n}\}\subset \mathbb{R},$ one can select a subsequence $\{\gamma'_{n}\}$ such that
for  any $\epsilon>0,$ there exists
 $N_{1}(\epsilon)\in \mathbb{N},$ when $n>N_{1},$ it holds
 $$\| u(\cdot+\gamma'_{n})-\tilde{u}(\cdot)\|_{W^p}<\frac{\epsilon}{2}\quad \mathrm{and} \quad \|\tilde{u}(\cdot-\gamma'_{n})-u(\cdot)\|_{W^p}<\frac{\epsilon}{2}.$$
Meanwhile,   there is a subsequence $\{\gamma_{n}\}$ of $\{\gamma'_{n}\}$ and a number
$N_{2}(\epsilon)\in \mathbb{N}$ such that, when $n>N_{2},$
$$
\|v(\cdot+\gamma_{n})-\tilde{v}(\cdot)\|_{W^p}<\frac{\epsilon}{2} \quad \mathrm{and} \quad\|\tilde{v}(\cdot-\gamma_{n})-v(\cdot)\|_{W^p}<\frac{\epsilon}{2}.$$
Take $N_{0}=\max(N_{1},N_{2}),$ for $n>N_{0},$ we have
$$\| u(\cdot+\gamma_{n})-\tilde{u}(\cdot)\|_{W^p}<\frac{\epsilon}{2} \quad \mathrm{and} \quad\| v(\cdot+\gamma_{n})-\tilde{v}(\cdot)\|_{W^p}<\frac{\epsilon}{2}.$$
Consequently, we arrive at, for $n>N_0$,
\begin{align*}
   &\| (u(\cdot+\gamma_{n})+v(\cdot+\gamma_{n}))-\tilde{u}(\cdot)-\tilde{v}(\cdot)\|_{W^p}\\
   \leq&\| u(\cdot+\gamma_{n})-\tilde{u}(\cdot)\|_{W^p}+\| v(\cdot+\gamma_{n})-\tilde{v}(\cdot)\|_{W^p}<\epsilon.
\end{align*}

Similarly,  for $n>N_0$, one can gain
$$\|\tilde{u}(\cdot-\gamma_{n})+\tilde{v}(\cdot-\gamma_{n})-u(\cdot)-v(\cdot)\|_{W^p}<\epsilon.$$
So $u+v\in W^p_{AA}(\mathbb{R},\mathbb{X})$ is proved.  The proof of $\alpha u\in W^p_{AA}(\mathbb{R},\mathbb{X})$ is trivial and we will omit it here. The proof of Lemma \ref{lem1} is complete.
\end{proof}
\end{lemma}

The following proposition shows that the Weyl automorphic function is translation invariant.
\begin{proposition}
If $u\in W^p_{AA}(\mathbb{R},\mathbb{X})$,  then $u(\cdot+\alpha)\in W^p_{AA}(\mathbb{R},\mathbb{X})$ for each $\alpha\in \mathbb{R}$.
\end{proposition}
\begin{proof}
Since $u\in W^p_{AA}(\mathbb{R},\mathbb{X}),$ for any sequence $\{\gamma'_{n}\}\subset \mathbb{R},$ one can choose a subsequence $\{\gamma_{n}\}$ such that
$$\lim\limits_{n\rightarrow\infty}\|u(\cdot+\gamma_{n})-\tilde u(\cdot)\|^p_{W^p}=0\quad
\mathrm{and}
\quad \lim\limits_{n\rightarrow\infty}\|\tilde u(\cdot-\gamma_{n})-u(\cdot)\|^p_{W^p}=0.$$

Thereupon, by Lemma \ref{lem21}, we have
$$\lim\limits_{n\rightarrow\infty}\|u(\cdot+\alpha+\gamma_{n})-\tilde u(\cdot+\alpha)\|^p_{W^p}=0\quad
\mathrm{and}
\quad \lim\limits_{n\rightarrow\infty}\|\tilde u(\cdot+\alpha-\gamma_{n})-u(\cdot+\alpha)\|^p_{W^p}=0.$$

Consequently, $u(\cdot+\alpha)\in W^p_{AA}(\mathbb{R},\mathbb{X}).$
The proof is completed.
\end{proof}

We state and prove the composition theorem of Weyl almost automorphic functions as follows.

 \begin{theorem}\label{aaa1}
If $f:\mathbb{X}\rightarrow \mathbb{X}$ satisfies the Lipschitz condition and $x\in W^p_{AA}(\mathbb{R},\mathbb{X}),$ then $f(x(\cdot))$ belongs to $W^p_{AA}(\mathbb{R},\mathbb{X}).$
\end{theorem}
\begin{proof}
It is easy to see that $f(x(\cdot))\in W^p_B(\mathbb{R},\mathbb{X})$. For any sequence  $(\gamma_{n})_{n\in \mathbb{N}}\subset \mathbb{R}$, note the fact that
\begin{align*}
\|f(x(t+\gamma_{n}))-f(\tilde{x}(t))\|_{W^p}=&\bigg(\lim\limits_{h\rightarrow\infty} \sup\limits_{\theta\in\mathbb{R}}\bigg(\frac{1}{h}\int_{\theta}^{\theta+h}\|f(x(t+\gamma_{n}))-f(\tilde{x}(t))\|^{p}dt\bigg)^{1/p}\\
&\leq \bigg(\lim\limits_{h\rightarrow\infty} \sup\limits_{\theta\in\mathbb{R}}\bigg(\frac{1}{h}\int_{\theta}^{\theta+h}L^p_f\|x(t+\gamma_{n})-\tilde{x}(t)\|^{p}dt\bigg)^{1/p}\\
&\leq L_f\|x(t+\gamma_{n})-\tilde{x}(t)\|_{W^p}
\end{align*}
and
 $$\|f(\tilde{x}(t-\gamma_{n}))-f({x}(t))\|_{W^p}\leq L_f\|\tilde{x}(t-\gamma_{n})-{x}(t)\|_{W^p},$$
 where $L_f$ is the Lipschitz constant and $\tilde{x}$ is mentioned in Definition \ref{zs}. Then, the conclusion follows immediately.
The proof is completed.
\end{proof}

\begin{proposition}\label{lemxl}
Let $\{\phi_n\}\subset W^p_{AA}(\mathbb{R},\mathbb{X})$ be a sequence such that $\lim\limits_{n\rightarrow\infty} \phi_n= \phi$  regarding the seminorm $\|\cdot\|_{W^p}$.
Then, $\phi\in W^p_{AA}(\mathbb{R},\mathbb{X})$.
\end{proposition}
\begin{proof}
Let $\{\gamma'_k\}\subset \mathbb{R}$ be an arbitrary sequence. By the
diagonal procedure we can choose a subsequence $\{\gamma_k\}$ such that
\begin{align}\label{abc}
  \lim\limits_{k\rightarrow \infty}\|\phi_n(\cdot+\gamma_k)-\tilde{\phi}_n(\cdot)\|_{W^p}=0
\end{align}
for each $n\in \mathbb{N} $. Noting that
\begin{align}\label{e22a}
&\|\tilde{\phi}_n(\cdot)-\tilde{\phi}_m(\cdot)\|_{W^p}\nonumber\\
\leq&\|\tilde{\phi}_n(\cdot)-\phi_n(\cdot+\gamma_k)\|_{W^p}+\|\phi_n(\cdot+\gamma_k)-\phi_m(\cdot+\gamma_k)\|_{W^p}+\|\phi_m(\cdot+\gamma_k)-\tilde{\phi}_m(\cdot)\|_{W^p}.
\end{align}
For any $\epsilon> 0$, based on the convergence of $\{\phi_n\}$,  we can find an integer $N>0$ such that when $n,m > N,$ for all $k\in \mathbb{N}$, it holds
$$\|\phi_n(\cdot+\gamma_k)-\phi_m(\cdot+\gamma_k)\|_{W^p}<\epsilon,$$
this combined with  \eqref{abc} and \eqref{e22a} implies that $\{\tilde{\phi}_n\}$ is a Cauchy sequence. Hence, by Theorem \ref{lem12} and Proposition \ref{lem23},  we can infer that there is a function $\tilde{\phi}\in W^p_B(\mathbb{R},\mathbb{X})$
such that $\lim\limits_{n\rightarrow \infty}\tilde{\phi}_n=\tilde{\phi}$.

Let us prove now that
$$\lim\limits_{k\rightarrow \infty}\lim \|\phi(\cdot + \gamma_k)-\tilde{\phi}(\cdot)\|_{W^p}=0$$
and
$$\lim\limits_{k\rightarrow \infty}\lim \|\tilde{\phi}(\cdot - \gamma_k)-\phi(\cdot)\|_{W^p}=0.$$

Indeed, for each $n\in \mathbb{N}$, we get
\begin{align*}
 & \|\phi(\cdot+\gamma_k)-\tilde{\phi}(\cdot)\|_{W^p}\\
 \leq& \|\phi(\cdot+\gamma_k)-\phi_n(\cdot+\gamma_k)\|_{W^p}+\|\phi_n(\cdot+\gamma_k)-\tilde{\phi}_n(\cdot)\|_{W^p}+\|\tilde{\phi}_n(\cdot)-\tilde{\phi}(\cdot)\|_{W^p}.
\end{align*}
Again, by the  convergence of $\{\phi_n\}$,
for any $\epsilon > 0,$ we can find some integer $N_0>0$ such that  for every  $k\in \mathbb{N}$,
$$\|\phi(\cdot+\gamma_k)-\phi_{N_0}(\cdot+\gamma_k)\|_{W^p}<\epsilon\quad \mathrm{and} \quad \|\tilde{\phi}_{N_0}(\cdot)-\tilde{\phi}(\cdot)\|_{W^p}<\epsilon.$$
Consequently,  for every  $k\in \mathbb{N}$, we have
$$\|\phi(\cdot+\gamma_k)-\tilde{\phi}(\cdot)\|<2\epsilon+\|\phi_{N_0}(\cdot+\gamma_k)-\tilde{\phi}_{N_0}(\cdot)\|_{W^p}.$$

Obviously, we can find some integer $N_1 = N(N_0)>0$
such that
$$\|\phi_{N_0}(\cdot+\gamma_k)-\tilde{\phi}_{N_0}(\cdot)\|_{W^p}<\epsilon$$
for every $k > N_1$.

Hence,
$$\|\phi(\cdot+\gamma_k)-\tilde{\phi}(\cdot)\|_{W^p}<3\epsilon$$
for $k>N_1$. That is, we have proven that
$$\|\phi(\cdot+\gamma_k)-\tilde{\phi}(\cdot)\|_{W^p}=0.$$

By the same method, one can show that
$$\lim\limits_{k\rightarrow \infty}\|\tilde{\phi}(\cdot-\gamma_k)-\phi(\cdot)\|_{W^p}=0.$$
 The proof is complete.
\end{proof}

\begin{theorem}\label{th2}
$((W^p_{AA}(\mathbb{R},\mathbb{X}),\|\cdot\|_{{W^p}})$ is a linear space, which is complete regarding the seminorm $\|\cdot\|_{{W^{p}}}$.
\end{theorem}
\begin{proof}
According to Lemma \ref{lem12}, $W^p_{B}(\mathbb{R},\mathbb{X})$  is complete regarding  the seminorm $\|\cdot\|_{{W^{p}}}$.
Since $W^p_{AA}(\mathbb{R},\mathbb{X})
\subset W^p_{B}(\mathbb{R},\mathbb{X})$, by Proposition \ref{lemxl}, $W^p_{AA}(\mathbb{R},\mathbb{X})$ is a closed subset of $W^p_{B}(\mathbb{R},\mathbb{X})$. Consequently, $W^p_{AA}(\mathbb{R},\mathbb{X})$ is complete regarding the seminorm $\|\cdot\|_{{W^{p}}}$. The proof is completed.
\end{proof}

The following is the definition of a Weyl almost automorphic function that depends on a  parameter.
\begin{definition} \label{def5}
A function $f:\mathbb{R}\times \mathbb{X}\rightarrow \mathbb{Y},(t,x)\mapsto f(t,x)$ with $f(\cdot,x)\in W^p_B(\mathbb{R},\mathbb{Y})$ for each $x\in \mathbb{X}$, is said to be $p$-th Weyl almost automorphic in $t\in \mathbb{R}$ uniformly in $x\in \mathbb{X}$ if for every bounded subset $\mathbb{B}\subset \mathbb{X}$ and every sequence $\{\gamma'_{n}\}\subset \mathbb{R}$, there exist a subsequence $\{\gamma_{n}\}\subset \{\gamma'_{n}\}$ and a function $\tilde{f}:\mathbb{R}\times \mathbb{X} \rightarrow \mathbb{Y}$ such that
$$\lim_{n\rightarrow\infty}\|f(\cdot+\gamma_{n},x)-\tilde{f}(\cdot,x)\|_{W^p}=0$$
 and
$$\lim_{n\rightarrow\infty}\|\tilde{f}(\cdot-\gamma_{n},x)-f(\cdot,x)\|_{W^p}=0$$
uniformly in $x\in \mathbb{B}$.
The collection of these functions will be  denoted by $W^p_{AA}(\mathbb{R}\times \mathbb{X},\mathbb{Y})$.
\end{definition}

\begin{lemma}\label{lm28}
If $f\in W^p_{AA}(\mathbb{R}\times \mathbb{X},\mathbb{Y})$ and there exists a constant $L_f>0$ such that $\|f(t,x)-f(t,y)\|_\mathbb{Y}\leq L_f\|x-y\|_\mathbb{Y}$ for all $t\in \mathbb{R}, x,y\in \mathbb{X}$, then $\| \tilde{f}(t,x)- \tilde{f}(t,y)\|_{W^p}\leq L_f\|x-y\|_{W^p}$,
where $\tilde{f}$ is mentioned in Definition \ref{def5}.
\end{lemma}
\begin{proof}For any sequence $\{\gamma_n\}$, since
\begin{align*}
&\| \tilde{f}(t,x)- \tilde{f}(t,y)\|_\mathbb{Y}\\
\leq &\| f(t+\gamma_{n},x)- \tilde{f}(t,x)\|_\mathbb{Y}+\| f(t+\gamma_{n},x)- f(t+\gamma_{n},y)\|_\mathbb{Y}+\| f(t+\gamma_{n},y)- \tilde{f}(t,y)\|_\mathbb{Y}\\
\leq &\| f(t+\gamma_{n},x)- \tilde{f}(t,x)\|_\mathbb{Y}+L_f\|x- y\|+\| f(t+\gamma_{n},y)- \tilde{f}(t,y)\|_\mathbb{Y},
\end{align*}
the conclusion of the lemma follows immediately. The proof is done.
\end{proof}

Now we state and prove the composition theorem for Weyl almost automorphic functions depending on a parameter.

 \begin{theorem}\label{lem212}
Under the conditions of Lemma \ref{lm28}, if $g\in W^p_{AA}(\mathbb{R},\mathbb{X})\cap \mathcal{L}^\infty(\mathbb{R},\mathbb{X})$, then $f(\cdot,g(\cdot))$ belongs to $W^p_{AA}(\mathbb{R},\mathbb{Y}).$
\end{theorem}
\begin{proof}
By the definition of the seminorm $\|\cdot\|_{W^p}$ and  the Lipschitz condition, one can easily see that $f(\cdot,g(\cdot))\in W^p_B(\mathbb{R},\mathbb{Y})$. For every sequence of $\{\gamma'_{n}\}\subset \mathbb{R},$  one can extract a subsequence $\{\gamma_{n}\}$ such that
for every $\epsilon>0$ and every bounded subset $\mathbb{B}\subset \mathbb{X}$ with $\{g(t): t\in \mathbb{R}\}\subset \mathbb{B}$, there exists a positive number $N(\epsilon,\mathbb{B})$ satisfying for $n>N,$
\begin{align}
 \| g(\cdot+\gamma_{n})-\tilde{g}(\cdot)\|_{W^p}<\epsilon,\quad \|f(\cdot+\gamma_{n},x)-\tilde{f}(\cdot,x)\|_{W^p}<\epsilon,\label{g1} \\
 \|\tilde{g}(\cdot-\gamma_{n})-g(\cdot)\|_{W^p}<\epsilon,\quad \|\tilde{f}(\cdot-\gamma_{n},x)-f(\cdot,x)\|_{W^p}<\epsilon,\label{g2}
\end{align}
for all  $x\in \mathbb{B}$.
Therefore, in view of \eqref{g1}, we infer that
\begin{align*}
&\|f(\cdot+\gamma_{n}, g(\cdot+\gamma_{n}))-\tilde{f}(\cdot,\tilde{g}(\cdot))\|_{W^p}\\
\leq&\|f(\cdot+\gamma_{n}, g(\cdot+\gamma_{n}))-f(\cdot+\gamma_{n},\tilde{g}(\cdot))\|_{W^p}+ \|f(\cdot+\gamma_{n}, \tilde{g}(\cdot))-\tilde{f}(\cdot,\tilde{g}(\cdot))\|_{W^p}\\
=&\bigg(\lim\limits_{h\rightarrow\infty} \sup\limits_{\theta\in\mathbb{R}}\frac{1}{h}\int_{\theta}^{\theta+h}\|f(\cdot+\gamma_{n}, g(\cdot+\gamma_{n}))-f(\cdot+\gamma_{n},\tilde{g}(\cdot))\|_\mathbb{Y}^{p}dt\bigg)^{1/p}\\
&+ \|f(\cdot+\gamma_{n}, \tilde{g}(\cdot))-\tilde{f}(\cdot,\tilde{g}(\cdot))\|_{W^p} \\
&\leq L_f\bigg(\lim\limits_{h\rightarrow\infty} \sup\limits_{\theta\in\mathbb{R}}\frac{1}{h}\int_{\theta}^{\theta+h}\|g(\cdot+\gamma_{n})-\tilde{g}(\cdot)\|_\mathbb{Y}^{p}dt\bigg)^{1/p}+ \|f(\cdot+\gamma_{n}, \tilde{g}(\cdot))-\tilde{f}(\cdot,\tilde{g}(\cdot))\|_{W^p}^{p}\\
<&(L_f+1)\epsilon,
\end{align*}
that is, $\lim\limits_{n\rightarrow\infty}\|f(\cdot+\gamma_{n}, g(\cdot+\gamma_{n}))-\tilde{f}(\cdot,\tilde{g}(\cdot))\|_{W^p}=0$.

Similarly, by Lemma \ref{lm28},  \eqref{g2}  and
the fact that
\begin{align*}
& \|\tilde{f}(\cdot-\gamma_{n},\tilde{g}(\cdot-\gamma_{n}))-f(\cdot,{g}(\cdot))\|_{W^p}\\
 \leq&\|\tilde{f}(\cdot-\gamma_{n},\tilde{g}(\cdot-\gamma_{n}))-\tilde{f}(\cdot-\gamma_{n},{g}(\cdot))\|_{W^p}+\|\tilde{f}(\cdot-\gamma_{n},{g}(\cdot))-f(\cdot,{g}(\cdot))\|_{W^p}\\
  \leq&L_f\|\tilde{g}(\cdot-\gamma_{n})-{g}(\cdot)\|_{W^p}+\|\tilde{f}(\cdot-\gamma_{n},{g}(\cdot))-f(\cdot,{g}(\cdot))\|_{W^p},
\end{align*}
one can obtain that
$$\lim\limits_{n\rightarrow\infty}\|\tilde{f}(\cdot-\gamma_{n},\tilde{g}(\cdot-\gamma_{n}))-f(\cdot,{g}(\cdot))\|_{W^p}=0.$$
The proof is completed.
\end{proof}

Before concluding this section, we give the definition of a working space that will be used in Section 4 of this paper and provide the proof of its completeness.

Let $$\mathbb{W}=\{y\in W^{p}_{AA}(\mathbb{R},\mathbb{X})\cap \mathcal{L}^{\infty}(\mathbb{R},\mathbb{X})\}$$
 with the norm $\|\cdot\|_{\mathbb{W}}:=\|\cdot\|_{\infty}$.

\begin{theorem}\label{lem41} The space $(\mathbb{W},\|\cdot\|_{\mathbb{W}})$ is a Banach one.
\end{theorem}
\begin{proof}
Let $\{\phi_{n}\}$ be  a Cauchy sequence in $\mathbb{W}$.
Since $(\mathcal{L}^{\infty}(\mathbb{R},\mathbb{X}),\|\cdot\|_{\infty})$ is a Banach space and $\{\phi_n\} \subset \mathcal{L}^{\infty}(\mathbb{R},\mathbb{X})$ it follows that there exists $\phi\in \mathcal{L}^{\infty}(\mathbb{R},\mathbb{X})$ such that
\begin{align*}
\lim\limits_{n\rightarrow\infty}\|\phi_{n}-\phi\|_{\infty}=0.
\end{align*}
 Hence, to show $(\mathbb{W},\|\cdot\|_{\mathbb{W}})$ is a Banach space, it suffices to show $\phi\in W^{p}_{AA}(\mathbb{R},\mathbb{X})\}$.
Noting that $\|\phi_{n}-\phi\|_{W^p}\leq\|\phi_{n}-\phi\|_{\infty}$, consequently,
\begin{equation*}
\lim\limits_{n\rightarrow\infty} \|\phi_{n}-\phi\|_{W^p}=0.
\end{equation*}
 By Proposition \ref{lemxl}, we conclude that $\phi\in W^p_{AA}(\mathbb{R},\mathbb{X})$.
 The proof is complete.
\end{proof}

\section{Weyl double-measure pseudo almost automorphic functions}
\setcounter{equation}{0}
\indent

In this section, we will introduce the concept of Weyl double-measure pseudo almost automorphic functions.

Let $\Sigma$  indicate the Lebesgue $\sigma$-field of $\mathbb{R}$ and  $\mathcal{M}_P$ be the set of all positive measures $\mu$ on $\Sigma$ meeting $\mu(\mathbb{R})=+\infty$ and $\mu([a,b])<+\infty$
 for all $a,b\in \mathbb{R}\ \ (a\leq b)$.

Two measures $\mu,\nu \in \mathcal{M}_P$ are said to be equivalent, denoted as $\mu\sim\nu$,   if and only if there exist constants $\alpha,\beta> 0$ and a bounded interval $A$ (eventually $\emptyset$) such that
$\alpha \nu(B) \leq \mu(B)\leq\beta\nu(B)$  for all $B\in \Sigma$ with $B\cap A=\emptyset$.

For $\mu\in \mathcal{M}_P$ and $\sigma\in \mathbb{R}$, we define the measure $\mu_\sigma$ by $\mu_\sigma(A)=\mu(\{a+\sigma, a\in A\})$ for $A\in \Sigma$.

To introduce the Weyl double-measure pseudo almost automorphy, we need the following two fundamental assumptions:
 \begin{itemize}
 \item [$(B_1)$]  Let $\mu,\nu \in \mathcal{M}_P$ meet
 $\limsup\limits_{l\rightarrow\infty}\frac{\mu([-l,l])}{\nu([-l,l])}<\infty.$
  \item [$(B_2)$] For $\mu \in \mathcal{M}_P$ and all $\omega\in \mathbb{R}$, there are $\rho > 0$ and a bounded interval $\mathcal{J}\subset\mathbb{R}$ such that
 $$\mu({a+\omega : a\in A})\leq\rho\mu(A),$$
 provided $A\in \Sigma$ meets $A\cap \mathcal{J}=\emptyset$.
\end{itemize}

\begin{lemma}\cite{dd1}\label{lem210}
Let $\mu\in \mathcal{M}_P$, then $\mu$ satisfies $(B_2)$ if and only if $\mu \sim\mu_\sigma$ for all $\sigma\in \mathbb{R}$.
\end{lemma}

\begin{definition}\label{def31}
 Let $\mu,\nu\in \mathcal{M}_P$. A function $\varphi\in W^p_{B}(\mathbb{R}, \mathbb{X})$ belongs to $\mathcal{O}^p(\mathbb{R},\mathbb{X},\mu,\nu)$ if and only if it satisfies
$$ \lim\limits_{l\rightarrow +\infty}\frac{1}{\nu([-l,l])}\int_{-l}^{l} \bigg(\lim\limits_{h\rightarrow +\infty}\frac{1}{h}\int_{\theta}^{\theta+h}\|\varphi(s)\|^pds \bigg)^{\frac{1}{p}}d\mu(\theta)=0.$$
\end{definition}

\begin{definition}\label{def31}
 Let $\mu,\nu\in \mathcal{M}_P$. A function $\phi\in W^p_{B}(\mathbb{R}\times\mathbb{X}, \mathbb{Y})$ belongs to $\mathcal{O}^p(\mathbb{R}\times\mathbb{X},\mathbb{Y},\mu,\nu)$ if and only if it satisfies
$$ \lim\limits_{l\rightarrow +\infty}\frac{1}{\nu([-l,l])}\int_{-l}^{l} \bigg(\lim\limits_{h\rightarrow +\infty}\frac{1}{h}\int_{\theta}^{\theta+h}\|\phi(s,x)\|_\mathbb{Y}^pds \bigg)^{\frac{1}{p}}d\mu(\theta)=0$$
uniformly in $x\in \mathbb{B}$, where $\mathbb{B}$ is any bounded subset of $\mathbb{X}$.
\end{definition}

\begin{definition}
 Let $\mu,\nu\in \mathcal{M}_P$. A function $\varphi\in  \mathcal{L}^\infty(\mathbb{R},\mathbb{X})$ belongs to $\mathcal{E}(\mathbb{R},\mathbb{X},\mu,\nu)$ if and only if it satisfies
$$\lim\limits_{l\rightarrow +\infty}\frac{1}{\nu([-l,l])}\int_{-l}^{l}\|\varphi(t)\|d\mu(t)=0.$$
\end{definition}

\begin{remark}
Let $BC(\mathbb{R},\mathbb{X})$ be the class of all bounded and continuous functions from $\mathbb{R}$ to $\mathbb{X}$, then it is obvious that $BC(\mathbb{R},\mathbb{X})\subset \mathcal{L}^\infty(\mathbb{R},\mathbb{X})$.
\end{remark}

\begin{theorem}
Let $\mu,\nu\in \mathcal{M}_P$ satisfy $(B_1)$. Then $\mathcal{O}^p(\mathbb{R},\mathbb{X}, \mu,\nu)$ is a linear space and
complete regarding seminorm $\|\cdot\|_{W^p}$.
\end{theorem}
\begin{proof}
It is easy to see that $\mathcal{O}^p(\mathcal{R},\mathbb{X}, \mu,\nu)$ is a linear space. To finish the proof, it suffices to prove  that $\mathcal{O}^p(\mathcal{R},\mathbb{X}, \mu,\nu)$ is closed in $W^p_{B}(\mathbb{R},\mathbb{X})$.

 Let $\{\varphi_n\}\subset \mathcal{O}^p(\mathbb{R},\mathbb{X}, \mu,\nu)$ such that $\lim\limits_{n\rightarrow \infty}\|\varphi_n-\varphi\|_{W^p}=0$. Then, we have for any $\epsilon>0$, there exists an $N>0$ such that for any $n\geq N$,
 \begin{align*}
\|\varphi(\cdot)-\varphi_n(\cdot)\|^p_{W^p}<\epsilon.
\end{align*}

 In addition, since $\varphi_n\in \mathcal{O}^p(\mathbb{R},\mathbb{X}, \mu,\nu),$ for the above $\epsilon>0$,  there exists  $L_0(n)>0$ such that for any $l\geq L_0$, it holds
\begin{align*}
  \frac{1}{\nu([-l,l])}\int_{-l}^{l}\bigg(\lim\limits_{h\rightarrow \infty}\frac{1}{h}\int_{\theta}^{\theta+h}\|\varphi_n(s)\|^pds \bigg)^{\frac{1}{p}} d\mu(\theta)<\epsilon.
\end{align*}

Consequently, for $l\geq L_0(N)$, by the fact that $(a+b)^{\frac{1}{p}}\leq  a^{\frac{1}{p}}+b^{\frac{1}{p}}$ for all $a,b\geq 0$ and $p\geq 1$,  we deduce that
 \begin{align*}
 & \frac{1}{\nu([-l,l])}\int_{-l}^{l}\bigg(\lim\limits_{h\rightarrow \infty} \frac{1}{h}\int_{\theta}^{\theta+h}\|\varphi(s)\|^pds \bigg)^{\frac{1}{p}}d\mu(\theta)\\
  \leq&  \frac{1}{\nu([-l,l])}\int_{-l}^{l}\bigg(2^{p-1}\lim\limits_{h\rightarrow \infty} \frac{1}{h}\int_{\theta}^{\theta+h}\|\varphi(s)-\varphi_N(s)\|^pds \bigg)^{\frac{1}{p}}d\mu(\theta)\\
  &+ \frac{1}{\nu([-l,l])}\int_{-l}^{l}\bigg(2^{p-1}\lim\limits_{h\rightarrow \infty} \frac{1}{h}\int_{\theta}^{\theta+h}\|\varphi_N(s)\|^pds \bigg)^{\frac{1}{p}} d\mu(\theta)\\
  \leq& 2^{\frac{p-1}{p}} \frac{1}{\nu([-l,l])}\int_{-l}^{l}\|\varphi(\cdot)-\varphi_N(\cdot)\|_{W^p}   d\mu(\theta)\\
  &+  2^{\frac{p-1}{p}}\frac{1}{\nu([-l,l])}\int_{-l}^{l}\bigg(\lim\limits_{h\rightarrow \infty} \frac{1}{h}\int_{\theta}^{\theta+h}\|\varphi_N(s)\|^pds \bigg)^{\frac{1}{p}} d\mu(\theta)\\
   \leq& 2^{\frac{p-1}{p}}\bigg[ \frac{\mu([-l,l])}{\nu([-l,l])}+1\bigg]\epsilon.
 \end{align*}
Hence, $ \varphi \in \mathcal{O}^p(\mathcal{R},\mathbb{X}, \mu,\nu)$. This finishes the proof.
\end{proof}

\begin{theorem}\label{gx}
Let $\mu,\nu\in \mathcal{M}_P$ and $\mathcal{J}$ be a bounded interval (eventually $\emptyset$). Assume that $(B_1)$ holds and
$\varphi\in W^p_{B}(\mathbb{R},\mathbb{X})$. Then the following assertions are equivalent:
\begin{itemize}
  \item [$(i)$] Function $\varphi\in \mathcal{O}^p(\mathbb{R},\mathbb{X},\mu,\nu)$.
  \item [$(ii)$]
  $\lim\limits_{l\rightarrow +\infty}\frac{1}{\nu([-l,l]\setminus \mathcal{J})}\int_{{[-l,l]\setminus \mathcal{J}}}\Big(\lim\limits_{h\rightarrow +\infty}\frac{1}{h}\int_{\theta}^{\theta+h}\|\varphi(s)\|^pds \Big)^{\frac{1}{p}}d\mu(\theta)=0.$
  \item [$(iii)$] For any $\epsilon>0$, $\lim\limits_{l\rightarrow \infty}\frac{\mu(\{t\in [-l,l]\setminus \mathcal{J}: \|\varphi(t)\|> \epsilon\})}{\nu([-l,l]\setminus \mathcal{J})}=0$.
\end{itemize}
\end{theorem}
\begin{proof}
Let $D=\nu(\mathcal{J})$ and $E(\theta)=\int_\mathcal{J}\big(\lim\limits_{h\rightarrow +\infty}\frac{1}{h}\int_{\theta}^{\theta+h}\|\varphi(s)\|^pds \big)^{\frac{1}{p}} d\mu(\theta)$ and $F=\mu(\mathcal{J})$.

$(i)\Rightarrow(ii)$. Since $\varphi\in W^p(\mathbb{R},\mathbb{X})$ and $\mathcal{J}$ is bounded, $E(\theta)<\infty$ for all $\theta\in \mathbb{R}$. Hence, from  the fact that $\nu(\mathbb{R})=\infty$ and
\begin{align*}
& \frac{1}{\nu([-l,l]\setminus \mathcal{J})}\int_{{[-l,l]\setminus \mathcal{J}}}\bigg(\lim\limits_{h\rightarrow +\infty}\frac{1}{h}\int_{\theta}^{\theta+h}\|\varphi(s)\|^pds \bigg)^{\frac{1}{p}}d\mu(\theta)\\
 = &\frac{1}{\nu([-l,l])-D}\bigg[\int_{{[-l,l]}}\bigg(\lim\limits_{h\rightarrow +\infty}\frac{1}{h}\int_{\theta}^{\theta+h}\|\varphi(s)\|^pds \bigg)^{\frac{1}{p}}d\mu(\theta)-E(\theta)\bigg]\\
  = &\frac{\nu([-l,l])}{\nu([-l,l])-D}\bigg[\frac{1}{\nu([-l,l])}\int_{{[-l,l]}}\bigg(\lim\limits_{h\rightarrow +\infty}\frac{1}{h}\int_{\theta}^{\theta+h}\|\varphi(s)\|^pds \bigg)^{\frac{1}{p}}d\mu(\theta)-\frac{E(\theta)}{\nu([-l,l])}\bigg],
\end{align*}
  we can infer that $(ii)$ is equivalent to
\begin{align*}
\lim\limits_{l\rightarrow\infty} \frac{1}{\nu([-l,l])}\int_{{[-l,l]}}\bigg(\lim\limits_{h\rightarrow +\infty}\frac{1}{h}\int_{\theta}^{\theta+h}\|\varphi(s)\|^pds \bigg)^{\frac{1}{p}}d\mu(\theta)=0,
\end{align*}
that is $(i)$.

$(iii)\Rightarrow(ii)$. Denote $D^\epsilon_l=\{t\in [-l,l]\setminus \mathcal{J}:\|\varphi(t)\|>\epsilon\}$ and $E^\epsilon_l=\{t\in [-l,l]\setminus \mathcal{J}:\|\varphi(t)\|\leq \epsilon\}$.

Suppose that $(iii)$ holds, that is,
\begin{align*}
  \lim\limits_{l\rightarrow \infty}\frac{\mu(D^\epsilon_l)}{\nu([-l,l]\setminus \mathcal{J})}=0.
\end{align*}
From the equality
\begin{align*}
&\int_{{[-l,l]\setminus \mathcal{J}}}\bigg(\lim\limits_{h\rightarrow +\infty}\frac{1}{h}\int_{\theta}^{\theta+h}\|\varphi(s)\|^pds \bigg)^{\frac{1}{p}}d\mu(\theta)\\
=& \int_{{D^\epsilon_l}}\bigg(\lim\limits_{h\rightarrow +\infty}\frac{1}{h}\int_{\theta}^{\theta+h}\|\varphi(s)\|^pds \bigg)^{\frac{1}{p}}d\mu(\theta)+\int_{{E^\epsilon_l}}\bigg(\lim\limits_{h\rightarrow +\infty}\frac{1}{h}\int_{\theta}^{\theta+h}\|\varphi(s)\|^pds \bigg)^{\frac{1}{p}}d\mu(\theta),
\end{align*}
we deduce that for $l$ sufficiently large
\begin{align*}
& \frac{1}{\nu([-l,l]\setminus \mathcal{J})}\int_{{[-l,l]\setminus \mathcal{J}}}\bigg(\lim\limits_{h\rightarrow +\infty}\frac{1}{h}\int_{\theta}^{\theta+h}\|\varphi(s)\|^pds \bigg)^{\frac{1}{p}}d\mu(\theta)\\
 \leq&2\|\varphi\|_{W^p}\frac{\mu(D^\epsilon_l)}{\nu([-l,l]\setminus \mathcal{J})}+\epsilon \frac{\mu(E^\epsilon_l)}{\nu([-l,l]\setminus \mathcal{J})}\\
 \leq&2\|\varphi\|_{W^p}\frac{\mu(D^\epsilon_l)}{\nu([-l,l]\setminus \mathcal{J})}+\epsilon \frac{\mu([-l,l]\setminus \mathcal{J})}{\nu([-l,l]\setminus \mathcal{J})}\\
 \leq&2\|\varphi\|_{W^p}\frac{\mu(D^\epsilon_l)}{\nu([-l,l]\setminus \mathcal{J})}+\epsilon \frac{\mu([-l,l])-F}{\nu([-l,l])-D}\\
 =&2\|\varphi\|_{W^p}\frac{\mu(D^\epsilon_l)}{\nu([-l,l]\setminus \mathcal{J})}+\epsilon \frac{\mu([-l,l])}{\nu([-l,l])}\frac{1-\frac{F}{\mu([-l,l])}}{1-\frac{D}{\nu([-l,l])}}.
\end{align*}
Since $\mu(\mathbb{R})=\nu(\mathbb{R})=\infty$, then for all $\epsilon>0$, there holds
\begin{align*}
   \limsup\limits_{l\rightarrow \infty}\frac{1}{\nu([-l,l]\setminus \mathcal{J})}\int_{{[-l,l]\setminus \mathcal{J}}}\bigg(\lim\limits_{h\rightarrow +\infty}\frac{1}{h}\int_{\theta}^{\theta+h}\|\varphi(s)\|^pds \bigg)^{\frac{1}{p}}d\mu(\theta)\leq C\epsilon,
\end{align*}
where $C=\lim\limits_{l\rightarrow \infty}\frac{\mu([-l,l])}{\nu([-l,l])}$ is a constant by $(B_1)$.

$(ii)\Rightarrow(iii)$. Assume that $(ii)$ holds. From the following inequality:
\begin{align*}
 & \frac{1}{\nu([-l,l]\setminus \mathcal{J})}\int_{{[-l,l]\setminus \mathcal{J}}}\bigg(\lim\limits_{h\rightarrow +\infty}\frac{1}{h}\int_{\theta}^{\theta+h}\|\varphi(s)\|^pds \bigg)^{\frac{1}{p}}d\mu(\theta)\\
  \geq&\frac{1}{\nu([-l,l]\setminus \mathcal{J})}\int_{D^\epsilon_l}\bigg(\lim\limits_{h\rightarrow +\infty}\frac{1}{h}\int_{\theta}^{\theta+h}\|\varphi(s)\|^pds \bigg)^{\frac{1}{p}}d\mu(\theta)\\
  \geq&\epsilon\frac{\mu(D^\epsilon_l)}{\nu([-l,l]\setminus \mathcal{J})},
\end{align*}
for $l$ sufficiently large, we obtain   $(iii)$.
The proof is done.
\end{proof}

\begin{theorem}\label{n2}
Let $\mu_i,\nu_i\in \mathcal{M}_P, i=1,2$ and $\mu_1\sim\mu_2, \nu_1 \sim \nu_2$. If $(B_1)$ holds, then  $\mathcal{O}^p(\mathbb{R},
\mathbb{X},\mu_1,\nu_1)=\mathcal{O}^p(\mathbb{R},\mathbb{X},
\mu_2,\nu_2)$.
\end{theorem}
\begin{proof}
Because $\mu_1\sim\mu_2, \nu_1 \sim \nu_2$, so, there are positive constants $\alpha_i, \beta_i, i=1,2$ such that
\begin{align*}
 \alpha_1\mu_1(A_1)\leq \mu_2(A_1)\leq \beta_1\mu_1(A_1)\, \text{ and }\, \alpha_2\nu_1(A_2)\leq \nu_2(A_2)\leq \beta_2\nu_1(A_2),\,\, A_1,A_2\in \Sigma.
\end{align*}

For each $\varphi\in \mathcal{O}^p(\mathbb{R},\mathbb{X},\mu_1,\nu_1)$ and any $\epsilon>0$, since $\sum$ is the Lebesgue $\sigma$-field, for $l$ sufficiently large, we obtain
\begin{align*}
 \frac{\alpha_1}{\beta_2}\frac{\mu_1(\{t\in [-l,l]\setminus \mathcal{J}: \|\varphi(t)\|> \epsilon\})}{\nu_1([-l,l]\setminus \mathcal{J})} \leq &\frac{\mu_2(\{t\in [-l,l]\setminus \mathcal{J}: \|\varphi(t)\|> \epsilon\})}{\nu_2([-l,l]\setminus \mathcal{J})}\\
  \leq &\frac{\beta_1}{\alpha_2}\frac{\mu_1(\{t\in [-l,l]\setminus \mathcal{J}: \|\varphi(t)\|> \epsilon\})}{\nu_1([-l,l]\setminus \mathcal{J})},
\end{align*}
which by Theorem \ref{gx} implies that  $\varphi\in \mathcal{O}^p(\mathbb{R},\mathbb{X},\mu_2,\nu_2)$. As a consequence, $\mathcal{O}^p(\mathbb{R},
\mathbb{X},\mu_1,\nu_1)\subseteq\mathcal{O}^p(\mathbb{R},\mathbb{X},
\mu_2,\nu_2)$. Similarly, we can show that $\mathcal{O}^p(\mathbb{R},\mathbb{X},\mu_1,\nu_1)\supseteq\mathcal{O}(\mathbb{R},\mathbb{X},
\mu_2,\nu_2)$. The proof is ended.
\end{proof}

\begin{theorem}
Let $\mu,\nu\in \mathcal{M}_P$ satisfy $(B_1)$. Then $\mathcal{O}^p(\mathbb{R},\mathbb{X}, \mu,\nu)$ is translation invariant.
\end{theorem}
\begin{proof}
For each $\varphi\in\mathcal{O}^p(\mathcal{R},\mathbb{X}, \mu,\nu)$ and $\omega\in \mathbb{R}$,   we infer that
\begin{align*}
  &\frac{1}{\nu([-l,l])}\int_{-l}^{l}\bigg(\lim\limits_{h\rightarrow +\infty}\frac{1}{h}\int_{\theta}^{\theta+h}\|\varphi(s+\omega)\|^pds \bigg)^{\frac{1}{p}}d\mu(\theta)\\
  = & \frac{1}{\nu([-l,l])}\int_{-l}^{l}\bigg(\lim\limits_{h\rightarrow +\infty}\frac{1}{h}\int_{\theta+\omega}^{\theta+\omega+h}\|\varphi(s)\|^pds \bigg)^{\frac{1}{p}}d\mu(\theta)\\
  =&\frac{1}{\nu([-l,l])}\int_{-l}^{l}\bigg(\lim\limits_{h\rightarrow +\infty}\frac{1}{h}\int_{\theta}^{\theta+h}\|\varphi(s)\|^pds \bigg)^{\frac{1}{p}}d\mu_{-\omega}(\theta)\\
  =&\frac{\nu_{-\omega}([-l,l)}{\nu([-l,l])}\frac{1}{\nu_{-\omega}([-l,l)}\int_{-l}^{l}\bigg(\lim\limits_{h\rightarrow +\infty}\frac{1}{h}\int_{\theta}^{\theta+h}\|\varphi(s)\|^pds \bigg)^{\frac{1}{p}}d\mu_{-\omega}(\theta).
\end{align*}
Due to the fact that $\mu\sim \mu_{-\omega}$ and $\nu\sim \nu_{-\omega}$ by Lemma \ref{lem210}, according to Theorem \ref{n2} and the
Lebesgue dominated convergence theorem, there holds:
\begin{align*}
 \lim\limits_{\gamma\rightarrow \infty} \frac{1}{\nu([-\gamma,\gamma])}\int_{-\gamma}^{\gamma} \|\varphi(t+\omega)\|^pd\nu(t)=0.
\end{align*}
The proof is ended.
\end{proof}

Similarly, one can easily show that
\begin{theorem}\label{lempy1}
Let $\mu,\nu\in \mathcal{M}_P$ satisfy $(B_1)$. Then $\mathcal{E}(\mathcal{R},\mathbb{X}, \mu,\nu)$ is translation invariant.
\end{theorem}

\begin{theorem}\label{lem01}
Let $\mu,\nu\in \mathcal{M}_P$ satisfy $(B_1)$ and $(B_2)$. Then $\mathcal{E}(\mathcal{R},\mathbb{X}, \mu,\nu)\subset \mathcal{O}^p(\mathcal{R},\mathbb{X}, \mu,\nu)$.
\end{theorem}
\begin{proof}
Let $\varphi\in \mathcal{O}(\mathcal{R},\mathbb{X}, \mu,\nu)$, then by the H\"{o}lder inequality and the Fubini theorem, it can be infered that
\begin{align}\label{lb}
&\frac{1}{\nu([-l,l])}\int_{-l}^{l}\bigg(\lim\limits_{h\rightarrow +\infty}\frac{1}{h}\int_{\theta}^{\theta+h}\|\varphi(s)\|^pds \bigg)^{\frac{1}{p}}d\mu(\theta)\nonumber\\
\leq&\frac{1}{\nu([-l,l])}\bigg(\int_{-l}^{l}d\mu(\theta)\bigg)^{1-\frac{1}{p}}\bigg(\int_{-l}^{l}\lim\limits_{h\rightarrow +\infty}\frac{1}{h}\int_{\theta}^{\theta+h}\|\varphi(s)\|^pds d\mu(\theta)\bigg)^{\frac{1}{p}}\nonumber\\
=&\frac{1}{\nu([-l,l])} (\mu([-l,l])^{1-\frac{1}{p}}\bigg(\int_{-l}^{l}\lim\limits_{h\rightarrow +\infty}\frac{1}{h}\int_{\theta}^{\theta+h}\|\varphi(s)\|^{p-1}\|\varphi(s)\|ds d\mu(\theta)\bigg)^{\frac{1}{p}}\nonumber\\
\leq&\frac{1}{\nu([-l,l])} (\mu([-l,l])^{1-\frac{1}{p}}\|\varphi\|^{p-1}_\infty\bigg(\int_{-l}^{l}\lim\limits_{h\rightarrow +\infty}\frac{1}{h}\int_{\theta}^{\theta+h}\|\varphi(s)\|ds d\mu(\theta)\bigg)^{\frac{1}{p}}\nonumber\\
=&\frac{1}{\nu([-l,l])} (\mu([-l,l])^{1-\frac{1}{p}}\|\varphi\|^{p-1}_\infty\bigg(\int_{-l}^{l}\lim\limits_{h\rightarrow +\infty}\frac{1}{h}\int_{0}^{h}\|\varphi(s+\theta)\|ds d\mu(\theta)\bigg)^{\frac{1}{p}}\nonumber\\
=& \bigg(\frac{\mu([-l,l])}{\nu([-l,l])}\bigg)^{1-\frac{1}{p}}\|\varphi\|^{p-1}_\infty\bigg(\lim\limits_{h\rightarrow +\infty}\frac{1}{h}\int_{0}^{h}\frac{1}{\nu([-l,l])}\int_{-l}^{l}\|\varphi(s+\theta)\|d\mu(\theta)ds \bigg)^{\frac{1}{p}}.
\end{align}
Since $\mathcal{E}(\mathcal{R},\mathbb{X}, \mu,\nu)$ is translation invariant by Theorem \ref{lempy1}, we have
\begin{align*}
 \lim\limits_{l\rightarrow\infty} \frac{1}{\nu([-l,l])}\int_{-l}^{l}\|\varphi(s+\theta)\|d\mu(\theta)=0.
\end{align*}
Hence, from \eqref{lb} and the Lebesgue dominated convergence, it follows that
\begin{align*}
\lim\limits_{l\rightarrow\infty}\frac{1}{\nu([-l,l])}\int_{-l}^{l}\bigg(\lim\limits_{h\rightarrow +\infty}\frac{1}{h}\int_{\theta}^{\theta+h}\|\varphi(s)\|^pds \bigg)^{\frac{1}{p}}d\mu(\theta)=0,
\end{align*}
which indicates that $\varphi\in \mathcal{O}^p(\mathcal{R},\mathbb{X}, \mu,\nu)$. This finishes the proof.
\end{proof}

We are currently in a situation where we are able to put forward the concept of Weyl double-measure pseudo almost automorphic functions.
\begin{definition} \label{def1}
Let $\mu,\nu\in \mathcal{M}_P$. A  function $\phi\in W^p_{B}(\mathbb{R},\mathbb{X})$ is called to be $p$-th Weyl $(\mu,\nu)$-pseudo almost automorphic  if  it can be rewritten as
$$\phi = \phi_1 + \phi_2,$$
where $\phi_1\in W^p_{AA}(\mathbb{R},\mathbb{X})$ and $\phi_2\in \mathcal{O}^p(\mathbb{R},\mathbb{X},\mu,\nu)$. The collection of such functions will be denoted by $W^p_{PAA}(\mathbb{R}, \mathbb{X},\mu,\nu)$.
\end{definition}

From Definitions \ref{zs} and \ref{def1},  we have
$$W^p_{AA}(\mathbb{R},\mathbb{X})\subset W^p_{PAA}(\mathbb{R},\mathbb{X},\mu,\nu)\subset W^p_B(\mathbb{R},\mathbb{X}).$$

From Definition \ref{def1} and Theorem \ref{n2}, we have
\begin{theorem}\label{n3}
Let $\mu_i,\nu_i\in \mathcal{M}_P, i=1,2$ and $\mu_1\sim\mu_2, \nu_1 \sim \nu_2$. If $(B_1)$ holds, then  $W^p_{PAA}(\mathbb{R},\mathbb{X},\mu_1,\nu_1)=W^p_{PAA}(\mathbb{R},\mathbb{X},\mu_2,\nu_2)$.
\end{theorem}

We give the definition of Weyl double-measure pseudo almost automorphic functions that depend on a parameter as follows.
\begin{definition}\label{def2}
Let $\mu,\nu\in \mathcal{M}_P$. A  function $\varphi \in W^p_{B}(\mathbb{R}\times \mathbb{X},\mathbb{Y})$ is called  $p$-th Weyl $(\mu,\nu)$-pseudo almost automorphic in $t\in\mathbb{R}$  uniformly in $x\in \mathbb{X}$,    if  it can be rewritten as
$$\varphi = \varphi_1 + \varphi_2,$$
where $\varphi_1\in W^p_{AA}(\mathbb{R}\times \mathbb{X}, \mathbb{Y}, \mu, \nu)$ and $\varphi_2\in \mathcal{O}^p(\mathbb{R}\times \mathbb{X},\mathbb{Y},\mu,\nu)$. The collection of such functions will be denoted by $W^p_{PAA}(\mathbb{R}\times \mathbb{X}, \mathbb{Y},\mu,\nu)$.
\end{definition}

\begin{example}\label{ex1a2}
Take the Radon-Nikodym derivatives of $\mu$ and $\nu$ are $\sin t+2$ and $e^{|t|}+3+2\sin^2 t$, respectively. Then, we have
\begin{align*}
 \limsup\limits_{l\rightarrow\infty}\frac{\mu([-l,l])}{\nu([-l,l])}=\limsup\limits_{l\rightarrow\infty}\frac{\int_{-l}^{l}(\sin t+2)\mathrm{d}t}{\int_{-l}^{l}(e^{|t|}+3+2\sin^2 t)\mathrm{d}t}<\infty
\end{align*}
and
\begin{align*}
  \mu_\omega(A)=\int_{\omega+A}(\sin t+2)\mathrm{d}t\leq 3 \mu (A)
\end{align*}
for all $\omega\in \mathbb{R}$ and $A\in \Sigma$. Hence,
  conditions $(B_1)$ and $(B_2)$ are fulfilled.

Consider function
$$ \varphi(t)=\psi(t)+\alpha(t)+\gamma(t),$$ where $\psi\in W_{AA}^p(\mathbb{R}, \mathbb{R}), \gamma(t)=\arctan t, \alpha=e^{-|t|}$. It is easy to see that $\|\alpha\|_{W^p}=0, \|\gamma\|_{W^p}=\frac{\pi}{2}$ and $\alpha\in \mathcal{E}(\mathbb{R},\mathbb{R},\mu,\nu)\subset \mathcal{O}^p(\mathbb{R},\mathbb{R},\mu,\nu)$.

Since
\begin{align*}
\lim\limits_{l\rightarrow +\infty}\frac{1}{\nu([-l,l])}\int_{-l}^{l}|\arctan t |  d\mu(t)
 \leq \frac{\pi}{2}\lim\limits_{l\rightarrow\infty}\frac{\int_{-l}^{l}(\sin t+2)\mathrm{d}t}{\int_{-l}^{l}(e^{|t|}+3+2\sin^2 t)\mathrm{d}t}=0,
\end{align*}
 $\gamma\in \mathcal{E}(\mathbb{R},\mathbb{R},\mu,\nu)\subset \mathcal{O}^p(\mathbb{R},\mathbb{R},\mu,\nu)$. Hence, $\varphi\in W^p{PAA}(\mathbb{R},\mathbb{R},\mu,\nu)$.

If we take $\varphi_1(t)=\psi(t),\varphi_2(t)=\alpha(t)+\gamma(t)$, then   $\varphi_1  \in W^p_{AA}(\mathbb{R},\mathbb{R})$ and $\varphi_2\in \mathcal{E}(\mathbb{R},\mathbb{R},\mu,\nu)$. Moreover, if we take $\varphi_1(t)=\psi(t)+ \alpha(t),\varphi_2(t)=\gamma(t)$, then   $\varphi_1  \in W^p_{AA}(\mathbb{R},\mathbb{R})$ and $\varphi_2\in \mathcal{E}(\mathbb{R},\mathbb{R},\mu,\nu)$. That is to say, the decomposition of $\varphi$ is not unique.

\end{example}

\begin{remark}
From Example \ref{ex1a2}, we see that $W^p_{AA}(\mathbb{R},\mathbb{X})\cap \mathcal{O}^p(\mathbb{R},\mathbb{X},\mu,\nu)\neq \emptyset$. Therefore, although  $W^p_{AA}(\mathbb{R},\mathbb{X})$ and $\mathcal{O}^p(\mathbb{R},\mathbb{X},\mu,\nu)$ are complete regarding the seminorm $\|\cdot\|_{W^p}$, we could not show that $W_{PAA}^p(\mathbb{R},\mathbb{X},\mu,\nu)$ is complete with respect to $\|\cdot\|_{W^p}$.
\end{remark}

\section{Weyl $(\mu,\nu)$-pseudo almost automorphic solution}
\setcounter{equation}{0}
\indent

In this section, we will respectively establish the existence and uniqueness of Weyl almost automorphic solutions and Weyl $(\mu,\nu)$-pseudo almost automorphic solutions for the following abstract semilinear differential equation
\begin{align}\label{e1}
  x'(t)=A(t)x(t)+f(t,x(t)), \,\, t\in \mathbb{R},
\end{align}
where for each $t \in \mathbb{R}$, $A(t): D(A(t))\subset \mathbb{X}\rightarrow \mathbb{X}$  is a closed and densely defined linear operator  satisfying the so-called Acquistapace
and Terreni conditions \cite{at1}:
\begin{itemize}
  \item [$(S_1)$] There are constants $\lambda_0\geq 0, \theta\in (\frac{\pi}{2},\pi)$ and $K_1\geq 0$ such that $\Sigma_\theta \cup\{0\}\subset \rho(A(t)-\lambda_0)$ and for all $x\in \Sigma_0\cup \{0\}, t\in \mathbb{R},$
  $$\|R(\lambda, A(t)-\lambda_0)\| \leq \frac{K_1}{1+|\lambda|}.$$
  \item  [$(S_2)$] There are constants $K_2\geq 0$ and $\alpha,\beta\in (0,1]$ with $\alpha+\beta>1$ such that for all $\lambda\in \Sigma_\theta$ and $t,s\in \mathbb{R}$,
  \begin{align*}
    \|(A(t)-\lambda_0)R(\lambda,A(t)-\lambda_0)[R(\lambda_0,A(t))-R(\lambda_0,A(s))\| \leq \frac{K_2|t-s|^\alpha}{|\lambda|^\beta},
  \end{align*}
\end{itemize}
where $R(\lambda,L)=(\lambda I-L)^{-1}$ for all $\lambda\in \rho(L),  \Sigma_\theta= \{\lambda\in \mathbb{F}/\{0\}: |\arg \lambda|\leq \theta\}$.

Under assumptions $(S_1)$ and $(S_2)$, in view of \cite{at2}, we see that $A(t)$  generates a unique evolution family $\{\mathcal{U}(t,s)\}_{t\geq s}$ satisfying $\mathcal{U}(t,s)x\subset D(A(t))$
for all $t\geq s$ with the following properties:
\begin{itemize}
  \item [$(i)$] $\mathcal{U}(t,\theta)\mathcal{U}(\theta,s)=\mathcal{U}(t,s)$ and $\mathcal{U}(s,s)=I$ for all $t\geq \theta\geq s$.
  \item  [$(ii)$] The mapping $(t,s)\mapsto \mathcal{U}(t,s)x$ is continuous for all $x\in \mathbb{X}$ and $t\geq s$.
  \item  [$(iii)$] $\mathcal{U}(\cdot,s)\in C^1((s,\infty), L(X)), \frac{\partial \mathcal{U}}{\partial t} (t,s)=A(t)\mathcal{U}(t,s)$ and
  $$
  \|A(t)^\jmath \mathcal{U}(t,s\|\leq \jmath(t-s)^{-\jmath},\quad \mathrm{for} \quad 0<t-s<1 \quad \mathrm{and} \quad \jmath=0,1.
  $$
\end{itemize}

 We make the following assumptions:
\begin{itemize}
\item [$(H_1)$]  There exist constant $K\geq1,\lambda>0$ such that $\|\mathcal{U}(t,s)\|\leq K e^{-\lambda (t-s)}, t\geq s$.

\item [$(H_2)$] The evolution family  $\{\mathcal{U}(t,s)\}_{t\geq s}$ is bi-almost automorphic.

\item [$(H_3)$] Function $f: \mathbb{R}\times \mathbb{X} \rightarrow \mathbb{X}$ satisfies that   there are positive constants $L_f$ and $M_f$ such that  for all $x,y\in \mathbb{X}$ and $t\in\mathbb{R}$,
\begin{align*}
\|f(t,x)-f(t,y)\|\leq L_f\|x-y\| \quad \text{and}\quad
\|f(t,x)\|\leq M_f(1+\|x\|).
\end{align*}

 \item [$(H_4)$] Function $f\in W^p_{AA}(\mathbb{R}\times \mathbb{X},\mathbb{X})$ satisfies $(H_3)$.

\item [$(H_5)$] Function $f\in W^p_{PAP}(\mathbb{R}\times \mathbb{X},\mathbb{X})$ can be decomposed as $f=f_1+f_2$, where $f_1\in W^p_{AP}(\mathbb{R}\times \mathbb{X},\mathbb{X})$ and $f_2\in \mathcal{O}^p(\mathbb{R}\times \mathbb{X},\mathbb{X})$, and   there are constants $L_f$ and $M_f$ such that  for all $x,y\in \mathbb{X}$ and $t\in\mathbb{R}$,
\begin{align*}
&\|f(t,x)-f(t,y)\|\leq L_f\|x-y\| \quad \text{and}\quad
\|f(t,x)\|\leq M_f(1+\|x\|),\\
&\|f_1(t,x)-f_1(t,y)\|\leq L_f\|x-y\| \quad \text{and}\quad
\|f_1(t,x)\|\leq M_f(1+\|x\|).
\end{align*}

\item [$(H_6)$]The constant $\Theta:= \frac{K L_f}{\lambda}<1$, where $K$ is mentioned in $(H_1)$.
\end{itemize}

\begin{definition}\label{def41}
A mild solution of equation \eqref{e1} is a  function $x\in \mathcal{L}^\infty( \mathbb{R},\mathbb{X})$ that satisfies
\begin{align*}
  x(t)=\mathcal{U}(t,s)x(s)+\int_s^t\mathcal{U}(t,\theta)f(\tau,x(\theta))d \theta,\,\, t\geq s,\, s\in \mathbb{R}.
\end{align*}
\end{definition}

If $(H_1)$ holds,  then we have
\begin{align*}
  x(t)=\int_{-\infty}^t\mathcal{U}(t,\theta)f(\theta,x(\theta))d \theta,\,\, t\geq s,\, s\in \mathbb{R}.
\end{align*}

\begin{lemma}\label{lemab}
Let $(H_1)$  and $(H_2)$ hold. Then  it holds
\begin{align*}
 \|\tilde{\mathcal{U}}(t,s)\|\leq K e^{-\lambda(t-s)},\,\, t\geq s.
\end{align*}
\end{lemma}
\begin{proof} Note the fact that for any sequence $\{\gamma_n\}\subset \mathbb{R}$, there holds
\begin{align*}
&\| \tilde{\mathcal{U}}(t,s)\|\leq \|\mathcal{U}(t+\gamma_{n},s+\gamma_{n})-\tilde{\mathcal{U}}(t,s)\|+\|\mathcal{U}(t+\gamma_{n},s+\gamma_{n})\|\\
\leq &\|\mathcal{U}(t+\gamma_{n},s+\gamma_{n})-\tilde{\mathcal{U}}(t,s)\|+K e^{-\lambda(t-s)}.
\end{align*}
The conclusion of the lemma follows immediately. This finishes the proof.
\end{proof}

\begin{lemma}\label{lem42}
Let $(H_4)$  hold. If  $x\in \mathbb{W}$, then  $f(\cdot,x(\cdot))\in \mathbb{W}$.
\end{lemma}
\begin{proof}According to $(H_3)$, we have for all $t\in \mathbb{R}$,
\begin{align*}
\|f(t,x(t))\| \leq M_f(1+\|x(t)\|)\leq M_f(1+\|x\|_\mathbb{W}).
\end{align*}
 The above inequality implies $f(\cdot,x(\cdot))\in \mathcal{L}^\infty(\mathbb{R},\mathbb{H})$.
  By Theorem \ref{lem212}, we see that $f(\cdot,x(\cdot))\in W^p_{AA}(\mathbb{R},\mathbb{X})$. Hence, we conclude that $f(\cdot,x(\cdot)) \in \mathbb{W}$. The proof is complete.
\end{proof}

\begin{theorem}\label{th41}
Let conditions $(S_1)$, $(S_2)$, $(H_1)$, $(H_3)$ and $(H_6)$ be fulfilled.
Then equation  \eqref{e1} has a mild unique solution in $\mathcal{L}^\infty(\mathbb{R},\mathbb{X})$.
\end{theorem}

\begin{proof}
Consider an operator $\mathcal{T}:\mathcal{L}^\infty(\mathbb{R},\mathbb{X})\rightarrow \mathcal{L}^\infty(\mathbb{R},\mathbb{X})$ defined by
\begin{align}
(\mathcal{T} x)(t)=\int_{-\infty}^tT(t-s)f(s,x(s))\mathrm{d}s,
\end{align}
where $x\in \mathcal{L}^\infty(\mathbb{R},\mathbb{X})$.

Firstly, we will prove that the operator $\mathcal{T}$ is well defined.
 In fact, for every $x \in \mathcal{L}^\infty(\mathbb{R},\mathbb{X})$, by $(H_1)$ and $(H_3)$, one has
 \begin{align}\label{e32}
\|(\mathcal{T} x)(t)\|\leq&\bigg\|\int_{-\infty}^t\mathcal{U}(t,s)f(s,x(s))\mathrm{d}s\bigg\|\nonumber\\
\leq&K\int_{-\infty}^te^{-\lambda(t-s)}\|f(s,x(s))\|\mathrm{d}s\nonumber\\
\leq&K\frac{1}{\lambda} M_f (1+\|x\|_\infty).
\end{align}
Hence, $\mathcal{T} x\in \mathcal{L}^\infty(\mathbb{R},\mathbb{X})$.

Then, we will show that $\mathcal{T}$ is a contraction mapping. For every $x,y\in \mathcal{L}^\infty(\mathbb{R},\mathbb{X})$, according to $(H_1)$ and $(H_3)$, we can infer that
\begin{align*}
 \|(\mathcal{T} x)(t)-(\mathcal{T} y)(t)\|
\le&\bigg\|\int_{-\infty}^t\mathcal{U}(t,s)\big(f(s,x(s))-f(s,y(s))\big)\mathrm{d}s\bigg\|\\
\le&K L_f\int_{-\infty}^te^{-\lambda(t-s)}
\|x(s)-y(s)\| \mathrm{d}s\\
\le& K \frac{1}{\lambda} L_f\|x-y\|_{\infty},
\end{align*}
which yields that
\begin{align*}
\|\mathcal{T} x-\mathcal{T} y\|_{\mathbb{W}}\le&\Theta\|x-y\|_{\mathbb{W}}.
\end{align*}
By $(H_6)$, $\mathcal{T}$ is a contraction. Hence, $\mathcal{T}$ has a unique fixed  point $x \in \mathcal{L}^\infty(\mathbb{R},\mathbb{X})$. The proof is complete.
\end{proof}

\begin{lemma}\label{lem43} Assume that $(H_1)$ and $(H_2)$ hold. If $x\in \mathbb{W}$,  then   function $\Phi :\mathbb{R}\rightarrow \mathbb{X}$ defined by
\begin{equation}\label{4.6}
\Phi(t)=\int_{-\infty}^t\mathcal{U}(t,\theta)x(t)d \theta,\,\, t\geq s,\, s\in \mathbb{R}
\end{equation}
belongs to $\mathbb{W}$.
\end{lemma}
\begin{proof}

To start with, we show that $\Phi\in \mathcal{L}^\infty(\mathbb{R},\mathbb{X})$, in view of $(H_1)$,   we deduce that
\begin{align*}
 \|\Phi(t)\|
\leq&\int_{-\infty}^{t}\|\mathcal{U}(t,s)\| \|x(s)\| ds \\
\leq&\frac{K}{\lambda}\|x\|_{\mathbb{W}},
 \end{align*}
which implies that $\Phi\in \mathcal{L}^\infty(\mathbb{R},\mathbb{H})$.

Next, we will show that $\Phi x\in W^p_{AA}(\mathbb{R},\mathbb{X})$. By virtue of $x\in \mathbb{W}$ and condition $(H_2)$,  for every sequence $\{\gamma'_n\}\subset \mathbb{R}$, we can select a subsequence $\{\gamma_{n}\}\subset \{\gamma'_n\}$ such that
 \begin{align}
 \lim\limits_{n\rightarrow\infty}\|x(\cdot+\gamma_{n})-\tilde{x}(\cdot)\|_{W^p}=0, \label{eq1}\\
 \lim\limits_{n\rightarrow\infty}\|\tilde{x}(\cdot-\gamma_{n})-x(\cdot)\|_{W^p}=0, \label{eq2}
 \end{align}
 and that for each $t\in \mathbb{R}$,
 \begin{align}
  \lim\limits_{n\rightarrow\infty}\|\mathcal{U}(t+\gamma_{n},s+\gamma_{n})x-\tilde{\mathcal{U}}(t,s)\|=0, \label{eq3}\\
   \lim\limits_{n\rightarrow\infty}\|\tilde{\mathcal{U}}(t-\gamma_{n},s-\gamma_{n})x-\mathcal{U}(t,s)\|=0. \label{eq4}
 \end{align}

 Define a function $\tilde{\Phi}: \mathbb{R}\rightarrow \mathbb{X}$ by
\[
\tilde{\Phi}(t)=\int_{-\infty}^t\tilde{\mathcal{U}}(t,\theta)\tilde{x}(t)d \theta,\,\, t\geq s,\, s\in \mathbb{R}.
\]

Then, by the H\"{o}lder inequality, it can be infer that
\begin{align*}
 &\|\Phi(t+\gamma_{n})-\tilde{\Phi}(t)\|^p\\
 =&\bigg\|\int_{-\infty}^{t}\mathcal{U}(t+\gamma_{n},s+\gamma_{n})x(s+\gamma_{n})ds -\int_{-\infty}^{t}\tilde{\mathcal{U}}(t,s)\tilde{x}(s)ds\bigg\|^p\\
 \leq &2^{p-1}\bigg\|\int_{-\infty}^{t}\mathcal{U}(t+\gamma_{n},s+\gamma_{n})[x(s+\gamma_{n})-\tilde{x}(s)]ds\bigg\|^p\\
 &+2^{p-1}\bigg\|\int_{-\infty}^{t}[\mathcal{U}(t+\gamma_{n},s+\gamma_{n})-\tilde{\mathcal{U}}(t,s)]\tilde{x}(s)ds\bigg\|^p\\
 \leq &2^{p-1}K^p\bigg(\int_{-\infty}^{t}e^{-\lambda(t-s)}\|x(s+\gamma_{n})-\tilde{x}(s)\|ds\bigg)^p\\
 &+2^{p-1}\bigg(\int_{-\infty}^{t}\|\mathcal{U}(t+\gamma_{n},s+\gamma_{n})-\tilde{\mathcal{U}}(t,s)\|ds\bigg)^{\frac{p}{q}}\\
 &\times \int_{-\infty}^{t}\|\mathcal{U}(t+\gamma_{n},s+\gamma_{n})-\tilde{\mathcal{U}}(t,s)\|\|\tilde{x}(s))\|^pds \\
  \leq &2^{p-1}K^p\bigg(\frac{1}{\lambda}\bigg)^{\frac{p}{q}}\int_{-\infty}^{t}e^{-\lambda(t-s)}\|x(s+\gamma_{n})-\tilde{x}(s)\|^pds\\
 &+2^{p-1}\bigg(2K\frac{1}{\lambda}\bigg)^{\frac{p}{q}}\int_{-\infty}^{t}\|\mathcal{U}(t+\gamma_{n},s+\gamma_{n})-\tilde{\mathcal{U}}(t,s)\|\|\tilde{x}(s)\|^pds \\
 \leq &2^{p-1}K^p\bigg(\frac{1}{\lambda}\bigg)^{\frac{p}{q}}\int_{0}^{\infty}e^{-\lambda s}\|x(t-s+\gamma_{n})-\tilde{x}(t-s)\|^pds\\
 &+2^{p-1} \bigg(2K\frac{1}{\lambda}\bigg)^{\frac{p}{q}}\|\tilde{x}\|^p_\mathbb{W}\int_{-\infty}^{t}\|\mathcal{U}(t+\gamma_{n},s+\gamma_{n})-\tilde{\mathcal{U}}(t,s)\|ds \\
 :=&D_1(t)+D_2(t),
 \end{align*}
 where $\frac{1}{p}+\frac{1}{q}=1$.
 In the reasoning process of the above formula, we have used Lemma \ref{lemab} and the following fact:
\begin{align}\label{aa}
&\int_{-\infty}^{t}\|\mathcal{U}(t+\gamma_{n},s+\gamma_{n})-\tilde{\mathcal{U}}(t,s)\|ds\nonumber\\
\leq & \int_{-\infty}^{t}\|\mathcal{U}(t+\gamma_{n},s+\gamma_{n})\|ds + \int_{-\infty}^{t}\|\tilde{\mathcal{U}}(t,s)\|ds\nonumber\\
\leq & 2K\int_{-\infty}^{t}e^{-\lambda(t-s)} ds \nonumber\\
=& 2K\frac{1}{\lambda}.
\end{align}

By  Fubini's theorem, the Lebesgue's dominated convergence theorem and \eqref{eq1}, we infer that
\begin{align*}
 &\lim\limits_{n\rightarrow\infty}\lim\limits_{h\rightarrow \infty} \sup\limits_{\theta\in\mathbb{R}}\frac{1}{h}\int_{\theta}^{\theta+h}\|D_1(t)\|^p_{\mathbb{X}}dt\\
 \leq& 2^{p-1}K^p\bigg(\frac{1}{\lambda}\bigg)^{\frac{p}{q}}\int_{0}^{\infty}e^{-\lambda s}\lim\limits_{n\rightarrow\infty}\lim\limits_{h\rightarrow \infty} \sup\limits_{\theta\in\mathbb{R}}\frac{1}{h}\int_{\theta}^{\theta+h}\|x(t-s+\gamma_{n})-\tilde{x}(t-s)\|^p_{\mathbb{X}}dt ds\\
 =& 2^{p-1}K^p\bigg(\frac{1}{\lambda}\bigg)^{\frac{p}{q}}\int_{0}^{\infty}e^{-\lambda s}\lim\limits_{n\rightarrow\infty}\|x(\cdot+\gamma_{n})-\tilde{x}(\cdot)\|_{W^p}^p ds\\
 =&0.
 \end{align*}

 In view of \eqref{eq3} and \eqref{aa}, by the Lebesgue's dominated convergence theorem, we can get
 \begin{align*}
 \lim\limits_{n\rightarrow\infty}\lim\limits_{h\rightarrow \infty} \sup\limits_{\theta\in\mathbb{R}}\frac{1}{h}\int_{\theta}^{\theta+h}\| D_2(t)\|dt
 =0.
 \end{align*}

 Similarly, from \eqref{eq2} and \eqref{eq4}, one can get
 $$\lim\limits_{n\rightarrow\infty}\|\tilde{\Phi}(t-\gamma_{n})-\Phi(t)\|_{W^p}=0,$$
 which means that $\Phi\in \mathbb{W}$. The proof is complete.
\end{proof}

\begin{theorem}\label{th42} Suppose that $(S_1)$, $(S_2)$, $(H_1)$, $(H_2)$, $(H_4)$ and $(H_6)$ hold. Then equation \eqref{e1} has a unique $p$-th Weyl  almost automorphic mild solution in $\mathbb{W}$.
\end{theorem}
\begin{proof}Define an operator $\Psi:\mathbb{W}\rightarrow \mathbb{W}$ by
\begin{equation*}
(\Psi x)(t)=\int_{-\infty}^{t}\mathcal{U}(t,s)f(s,x(s))ds,\,\, x\in \mathbb{W}, \,\,t\in \mathbb{R}.
\end{equation*}
Since $(H_2)$ holds, in view of Lemma \ref{lem42}, we have $f(\cdot,x(\cdot))\in \mathbb{W}$ for $x\in \mathbb{W}$.

 Obviously, $\Psi$ is well-defined and $\Psi$  maps  $\mathbb{W}$ into $\mathbb{W}$ according to Lemma \ref{lem43}. We just have to show that $\Psi: \mathbb{W}\rightarrow \mathbb{W}$ is a contraction mapping. In fact, for any  $x,y\in \mathbb{W}$, by $(H_1)$ and $(H_4)$, we can deduce that
\begin{align*}
\|(\Psi x)(t)-(\Psi y)(t)\|=&\bigg\|\int_{-\infty}^{t}\mathcal{U}(t,s)(f(s,x(s))-f(s,y(s)))ds\bigg\|\\
\leq&\int_{-\infty}^{t}Ke^{-\lambda(t-s)}L_f\|x(s)-y(s)\|ds\\
\leq&\frac{KL_f}{\lambda}\|x-y\|_{\mathbb{W}}, \,\,t\in\mathbb{R},
\end{align*}
 which yields
$$\|\Psi x-\Psi y\|_{\mathbb{W}}\leq \Theta\|x-y\|_{\mathbb{W}}.$$
 Hence, in view of $(H_6)$, $\Psi$ is a contraction mapping from $\mathbb{Y}$ to $\mathbb{W}$. Therefore,   equation \eqref{e1} has a unique $p$-th Weyl almost automorphic mild solution. The proof is complete.
\end{proof}

\begin{theorem}\label{th43}
Let $\mu,\nu\in \mathcal{M}_P$ and  conditions $(B_1)$, $(B_2)$, $(S_1)$, $(S_2)$, $(H_1)$, $(H_2)$, $(H_5)$ and $(H_6)$ hold.
Then equation  \eqref{e1} has   a unique $p$-th Weyl $(\mu,\nu)$-pseudo almost automorphic mild solution.
\end{theorem}
\begin{proof}
According to Theorem \ref{th41}, we have that \eqref{e1} has   a unique solution $x \in \mathcal{L}^\infty(\mathbb{R},\mathbb{X})$  satisfies the following equation:
\begin{align}
 x(t)=\int_{-\infty}^t\mathcal{U}(t,s)f(s,x(s))\mathrm{d}s.
\end{align}
Moreover, according to Theorem  \ref{th42}, we can obtain that \eqref{e1} has   a unique $p$-th Weyl almost automorphic solution $x_1\in  \mathbb{W}$ satisfies the following equation:
\begin{align}
 x_1(t)=\int_{-\infty}^t\mathcal{U}(t,s)f_1(s,x_1(s))\mathrm{d}s.
\end{align}

Set $x-x_1=x_2$, then $x=x_1+x_2$. Obviously, in order to finish the proof, it suffices to show that $x_2\in \mathcal{O}^p(\mathbb{R},\mathbb{X})$.
Since $x\in \mathcal{L}^\infty(\mathbb{R},\mathbb{X})$ and $x_1\in \mathbb{W}$, $\|x_2\|_\infty<\infty$. In addition, by $(H_5)$, we have
$$\|f_2(t,x_1(t))\|\leq \|f(t,x_1(t))\|+\|f_1(t,x_1(t))\|\leq2 K(1+\|x_1\|)<\infty.$$

Hence, in view of H\"{o}lder's inequality and Fubini's theorem, it can be inferred that
\begin{align*}
  &\lim\limits_{l\rightarrow +\infty}\frac{1}{\nu([-l,l])}\int_{-l}^{l}\bigg(\lim\limits_{h\rightarrow +\infty}\frac{1}{h}\int_{\theta}^{\theta+h} \|x_2(t)\|^pdt \bigg)^{\frac{1}{p}}d\mu(\theta)\\
  \leq&\lim\limits_{l\rightarrow +\infty}\frac{1}{\nu([-l,l])}\int_{-l}^{l}\bigg(2^{p-1}\lim\limits_{h\rightarrow +\infty}\frac{1}{h}\int_{\theta}^{\theta+h} \bigg\|\int_{-\infty}^t\mathcal{U}(t,s)[f(s,x(s))\\
  &-f(s,x_1(s))]\mathrm{d}s\bigg\|^pdt \bigg)^{\frac{1}{p}}d\mu(\theta)+\lim\limits_{l\rightarrow +\infty}\frac{1}{\nu([-l,l])}\int_{-l}^{l}\bigg(2^{p-1}\lim\limits_{h\rightarrow +\infty}\frac{1}{h}\\
  &\times\int_{\theta}^{\theta+h} \bigg\|\int_{-\infty}^t\mathcal{U}(t,s)[f(s,x_1(s))-f_1(s,x_1(s))]\mathrm{d}s\bigg\|^pdt \bigg)^{\frac{1}{p}}d\mu(\theta)\\
  \leq&2^{\frac{p-1}{p}}ML_f\bigg(\frac{1}{\delta}\bigg)^{\frac{1}{q}}\lim\limits_{l\rightarrow +\infty}\frac{1}{\nu([-l,l])}\int_{-l}^{l}\bigg(\lim\limits_{h\rightarrow +\infty}\frac{1}{h}\int_{\theta}^{\theta+h}\int_{0}^\infty e^{-\delta s}
\|x_2(t-s)\|^p\mathrm{d}sdt \bigg)^{\frac{1}{p}}d\mu(\theta)\\
  &+ 2^{\frac{p-1}{p}}K\bigg(\frac{1}{\delta}\bigg)^{\frac{1}{q}}\lim\limits_{l\rightarrow +\infty}\frac{1}{\nu([-l,l])}\int_{-l}^{l}\bigg(\lim\limits_{h\rightarrow +\infty}\frac{1}{h}\\
  &\times\int_{\theta}^{\theta+h}\int_{0}^\infty e^{-\delta s}\|f_2(t-s,x_1(t-s))\|^p\mathrm{d}sdt \bigg)^{\frac{1}{p}}d\mu(\theta)\\
  \leq&2^{\frac{p-1}{p}}ML_f\bigg(\frac{1}{\delta}\bigg)^{\frac{1}{q}}\lim\limits_{l\rightarrow +\infty}\frac{1}{\nu([-l,l])}\int_{-l}^{l}\bigg(\int_{0}^\infty e^{-\delta s}\lim\limits_{h\rightarrow +\infty}\frac{1}{h}\int_{\theta}^{\theta+h}
\|x_2(t-s)\|^pdt\mathrm{d}s \bigg)^{\frac{1}{p}}d\mu(\theta)\\
  &+ 2^{\frac{p-1}{p}}K\bigg(\frac{1}{\delta}\bigg)^{\frac{1}{q}}\lim\limits_{l\rightarrow +\infty}\frac{1}{\nu([-l,l])}\int_{-l}^{l}\bigg(\int_{0}^\infty e^{-\delta s}\lim\limits_{h\rightarrow +\infty}\frac{1}{h}\\
  &\times\int_{\theta}^{\theta+h}\|f_2(t-s,x_1(t-s))\|^pdt\mathrm{d}s \bigg)^{\frac{1}{p}}d\mu(\theta)\\
  \leq&2^{\frac{p-1}{p}}ML_f\bigg(\frac{1}{\delta}\bigg)^{\frac{1}{q}}\lim\limits_{l\rightarrow +\infty}\frac{1}{\nu([-l,l])}\bigg(\int_{-l}^{l} d\mu(\theta)\bigg)^{1-\frac{1}{p}}\\
  &\times \bigg(\int_{-l}^{l}\int_{0}^\infty e^{-\delta s}\lim\limits_{h\rightarrow +\infty}\frac{1}{h}\int_{\theta}^{\theta+h}
\|x_2(t-s)\|^pdt\mathrm{d}sd\mu(\theta)\bigg)^{\frac{1}{p}}\\
  &+ 2^{\frac{p-1}{p}}K\bigg(\frac{1}{\delta}\bigg)^{\frac{1}{q}}\lim\limits_{l\rightarrow +\infty}\frac{1}{\nu([-l,l])}\bigg(\int_{-l}^{l} d\mu(\theta)\bigg)^{1-\frac{1}{p}}\\
  &\times \bigg(\int_{-l}^{l}\int_{0}^\infty e^{-\delta s}\lim\limits_{h\rightarrow +\infty}\frac{1}{h}\int_{\theta}^{\theta+h}\|f_2(t-s,x_1(t-s))\|^pdt\mathrm{d}sd\mu(\theta)\bigg)^{\frac{1}{p}}\\
   \leq&2^{\frac{p-1}{p}}ML_f\bigg(\frac{1}{\delta}\bigg)^{\frac{1}{q}}\|x_2\|^{1-\frac{1}{p}}_\infty\lim\limits_{l\rightarrow +\infty}\frac{(\mu([-l,l]) )^{1-\frac{1}{p}} }{\nu([-l,l])}\\
   &\times\bigg(\int_{-l}^{l}\int_{0}^\infty e^{-\delta s}\lim\limits_{h\rightarrow +\infty}\frac{1}{h}\int_{\theta}^{\theta+h}
\|x_2(t-s)\|dt\mathrm{d}sd\mu(\theta)\bigg)^{\frac{1}{p}}\\
  &+ 2^{\frac{p-1}{p}}K\bigg(\frac{1}{\delta}\bigg)^{\frac{1}{q}}\|f_2\|^{1-\frac{1}{p}}_\infty \lim\limits_{l\rightarrow +\infty}\frac{(\mu([-l,l]))^{1-\frac{1}{p}}}{\nu([-l,l])}\\
  &\times \bigg(\int_{-l}^{l}\int_{0}^\infty e^{-\delta s}\lim\limits_{h\rightarrow +\infty}\frac{1}{h}\int_{\theta}^{\theta+h}\|f_2(t-s,x_1(t-s))\|dt\mathrm{d}sd\mu(\theta)\bigg)^{\frac{1}{p}}\\
  \leq&2^{\frac{p-1}{p}}ML_f\bigg(\frac{1}{\delta}\bigg)^{\frac{1}{q}} \|x_2\|^{1-\frac{1}{p}}_\infty\lim\limits_{l\rightarrow +\infty}\frac{(\mu([-l,l]) )^{1-\frac{1}{p}}}{\nu([-l,l])}\\
   &\times\bigg(\int_{-l}^{l}\int_{0}^\infty e^{-\delta s}\bigg(\lim\limits_{h\rightarrow +\infty}\frac{1}{h}\int_{\theta}^{\theta+h}
\|x_2(t-s)\|^pdt\bigg)^{\frac{1}{p}}\mathrm{d}sd\mu(\theta)\bigg)^{\frac{1}{p}}\\
  &+ 2^{\frac{p-1}{p}}K\bigg(\frac{1}{\delta}\bigg)^{\frac{1}{q}}\|f_2\|^{1-\frac{1}{p}}_\infty \lim\limits_{l\rightarrow +\infty}\frac{(\mu([-l,l]))^{1-\frac{1}{p}}}{\nu([-l,l])}\\
  &\times \bigg(\int_{-l}^{l}\int_{0}^\infty e^{-\delta s}\bigg(\lim\limits_{h\rightarrow +\infty}\frac{1}{h}\int_{\theta}^{\theta+h}\|f_2(t-s,x_1(t-s))\|^pdt\bigg)^{\frac{1}{p}}\mathrm{d}s d\mu(\theta)\bigg)^{\frac{1}{p}}\\
 \leq&2^{\frac{p-1}{p}}ML_f\bigg(\frac{1}{\delta}\bigg)^{\frac{1}{q}} \|x_2\|^{1-\frac{1}{p}}_\infty\lim\limits_{l\rightarrow +\infty}\frac{(\mu([-l,l]) )^{1-\frac{1}{p}}}{\nu([-l,l])}\\
   &\times\bigg(\int_{0}^\infty e^{-\delta s}\int_{-l}^{l}\bigg(\lim\limits_{h\rightarrow +\infty}\frac{1}{h}\int_{\theta}^{\theta+h}
\|x_2(t-s)\|^pdt\bigg)^{\frac{1}{p}}d\mu(\theta)\mathrm{d}s\bigg)^{\frac{1}{p}}\\
  &+ 2^{\frac{p-1}{p}}K\bigg(\frac{1}{\delta}\bigg)^{\frac{1}{q}}\|f_2\|^{1-\frac{1}{p}}_\infty \lim\limits_{l\rightarrow +\infty}\frac{(\mu([-l,l]))^{1-\frac{1}{p}}}{\nu([-l,l])}\\
  &\times \bigg(\int_{0}^\infty e^{-\delta s}\int_{-l}^{l}\bigg(\lim\limits_{h\rightarrow +\infty}\frac{1}{h}\int_{\theta}^{\theta+h}\|f_2(t-s,x_1(t-s))\|^pdt\bigg)^{\frac{1}{p}} d\mu(\theta)\mathrm{d}s\bigg)^{\frac{1}{p}}\\
  \leq&2^{\frac{p-1}{p}}ML_f\bigg(\frac{1}{\delta}\bigg)^{\frac{1}{q}} \|x_2\|^{1-\frac{1}{p}}_\infty\lim\limits_{l\rightarrow +\infty}\bigg(\frac{\mu([-l,l])}{\nu([-l,l])}\bigg)^{1-\frac{1}{p}}\\
   &\times\bigg(\int_{0}^\infty e^{-\delta s}\lim\limits_{l\rightarrow +\infty}\frac{1}{\nu([-l,l])}\int_{-l}^{l}\bigg(\lim\limits_{h\rightarrow +\infty}\frac{1}{h}\int_{\theta}^{\theta+h}
\|x_2(t-s)\|^pdt\bigg)^{\frac{1}{p}}d\mu(\theta)\mathrm{d}s\bigg)^{\frac{1}{p}}\\
  &+ 2^{\frac{p-1}{p}}K\bigg(\frac{1}{\delta}\bigg)^{\frac{1}{q}}\|f_2\|^{1-\frac{1}{p}}_\infty \lim\limits_{l\rightarrow +\infty}\bigg(\frac{\mu([-l,l])}{\nu([-l,l])}\bigg)^{1-\frac{1}{p}}\\
  &\times \bigg(\int_{0}^\infty e^{-\delta s}\lim\limits_{l\rightarrow +\infty}\frac{1}{\nu([-l,l])}\int_{-l}^{l}\bigg(\lim\limits_{h\rightarrow +\infty}\frac{1}{h}\int_{\theta}^{\theta+h}\|f_2(t-s,x_1(t-s))\|^pdt\bigg)^{\frac{1}{p}} d\mu(\theta)\mathrm{d}s\bigg)^{\frac{1}{p}}\\
  =&0.
\end{align*}
Therefore, $x_2\in \mathcal{O}^p(\mathbb{R},\mathbb{X})$. The proof is complete.
\end{proof}

\section{Global exponential stability}
\setcounter{equation}{0}
\indent

In this section, we will demonstrate that the Weyl almost automorphic mild solutions and Weyl $(\mu,\nu)$-pseudo almost automorphic mild solutions of equation \eqref{e1} exhibit global exponential stability.

\begin{definition}
Let $x^*(t)$ be the  solution
of \eqref{e1}  with the initial value $x^*(t_0)$ and $x(t)$ be an arbitrary solution with the initial value $x(t_0)$. If there are positive numbers $N$ and $\kappa$ such
that
\[
\|x(t)-x^*(t)\|\leq N\|x(t_0)-x^*(t_0)\|e^{-\kappa(t-t_0)},\,\, \text{for all}\,\, t>t_0,
\]
then $x^*(t)$ is called to be globally exponentially stable.
\end{definition}

\begin{theorem}\label{th51}
 Assume that  conditions $(S_1)$, $(S_2)$, $(H_1)$,  $(H_2)$, $(H_4)$ and $(H_6)$ hold. Then, the unique $p$-th Weyl almost automorphic mild solution
of \eqref{e1} is globally exponentially stable.
\end{theorem}
\begin{proof}According to Theorem \ref{th42},  \eqref{e1} has a unique  $p$-th Weyl almost automorphic solution $x^*$.
Let $y=x-x^*$, then by Definition \ref{def41}, we have
\begin{align}\label{st}
y(t)&=\mathcal{U}(t,t_0)y(t_0)+\int_{t_0}^t\mathcal{U}(t,s)[f(s,x(s))-f(s,x^*(s))]\mathrm{d}s,\quad t \ge t_0,\, t_0\in\mathbb{R}.
\end{align}

Consider a linear function defined as
\begin{align*}
S(\theta)=\lambda-\theta-\Theta L_f.
\end{align*}

By $(H_6)$, $S(0)>0$. Obviously,   $S(\theta)=0$ has a unique positive root $\theta^*=\lambda-\Theta L_f$. Hence, we have $S(\theta)>$ for $\theta\in (0,\theta^*)$.
Take a positive number $\eta\in  (0,\min\{\lambda,\theta^*\})$. Then,  $S(\eta)>0$ and hence,
\begin{align*}
 \frac{\Theta L_f}{\lambda-\eta}<1.
\end{align*}
Set $N=\frac{\lambda}{L_f}$. Since $(H_4)$ and $K\geq 1$, we have $N>1$. Thus,
\begin{align*}
  \frac{1}{N}-\frac{L_f}{\lambda-\eta}<0.
\end{align*}

Now, we will prove that for arbitrary $\epsilon>0$,
\begin{align}\label{st1}
  \|y(t)\|< N(\|y(t_0)\|+\epsilon)e^{-\eta(t-t_0)},\,\, \text{for all}\,\, t\geq t_0.
\end{align}

Obviously, \eqref{st1} holds for $t=t_0$. If it does not hold for all $t>t_0$, then there must be a point $t_1>t_0$ such that
\begin{align}\label{st2}
  \|y(t)\|< & N(\|y(t_0)\|+\epsilon)e^{-\eta(t-t_0)},\,\, t_0\leq t<t_1,\nonumber\\
  \|y(t_1)\| = &N(\|y(t_0)\|+\epsilon)e^{-\eta(t_1-t_0)}.
\end{align}
From \eqref{st}, it follows that
\begin{align*}
\|y(t_1)\|=&\bigg\|\mathcal{U}(t_1,t_0)y(t_0)+\int_{t_0}^{t_1}\mathcal{U}(t_1,s)[f(s,x(s))-f(s,x^*(s))]\mathrm{d}s\bigg\|\nonumber\\
\leq& \|\mathcal{U}(t_1,t_0)y(t_0)\|^p+\bigg\|\int_{t_0}^{t_1}\mathcal{U}(t_1,s)[f(s,x(s))-f(s,x^*(s))]\mathrm{d}s\bigg\| \nonumber\\
\leq& K  \|y(t_0)\| e^{-\lambda(t_1-t_0)}+ K L_f\int_{t_0}^{t_1}e^{-\lambda(t_1-s)}\|x(s)-x^*(s)\|\mathrm{d}s\nonumber\\
\leq& K   (\|y(t_0)\|+\epsilon) e^{-\lambda(t_1-t_0)}+ K L_f N(\|y(t_0)\|+\epsilon) \int_{t_0}^{t_1}e^{-\lambda(t_1-s)}e^{-\eta(s-t_0)}\mathrm{d}s  \nonumber\\
\leq& K   (\|y(t_0)\|+\epsilon) e^{-\lambda(t_1-t_0)}+K L_f N(\|y(t_0)\|+\epsilon)e^{-\eta(t_1-t_0)} \int_{t_0}^{t_1}e^{-(\lambda-\eta)(t_1-s)}\mathrm{d}s  \nonumber\\
\leq& K   (\|y(t_0)\|+\epsilon) e^{-\lambda(t_1-t_0)}+  K L_f N(\|y(t_0)\|+\epsilon)e^{-\eta(t_1-t_0)} \frac{1}{\lambda-\eta}\big[1- e^{-(\lambda-\eta)(t_1-t_0)}\big]\\
 \leq& N(\|y(t_0)\|+\epsilon)e^{-\eta(t_1-t_0)}\bigg(\frac{K e^{-(\lambda-\eta)(t_1-t_0)}}{N} + \frac{K L_f}{\lambda-\eta}\big[1- e^{-(\lambda-\eta)(t_1-t_0)}\big]   \bigg)   \nonumber\\
   \leq& N(\|y(t_0)\|+\epsilon)e^{-\eta(t_1-t_0)}\bigg[K e^{-(\lambda-\eta)(t_1-t_0)}\bigg(\frac{1}{N}-\frac{L_f}{\lambda-\eta}\bigg) + \frac{K L_f}{\lambda-\eta} \bigg]   \nonumber\\
   <& N(\|y(t_0)\|^p+\epsilon)e^{-\eta(t_1-t_0)},
\end{align*}
which contradicts \eqref{st2}. Hence,  \eqref{st1} holds. The proof is complete.
\end{proof}
 Similarly, one can prove that
\begin{theorem}\label{thm3}
Let $\mu, \nu\in \mathcal{M}_P$. Assume that  conditions $(B_1)$, $(B_2)$, $(S_1)$, $(S_2)$, $(H_1)$,  $(H_2)$, $(H_5)$ and $(H_6)$ hold. Then, the unique $p$-th Weyl $(\mu,\nu)$-pseudo almost automorphic mild solution
of \eqref{e1} is globally exponentially stable.
\end{theorem}

\section{An example}
\setcounter{equation}{0}
\indent

In this section, we will invoke a concrete example to showcase that the hypotheses outlined in the main theorems are readily met.

Consider the following  partial differential equation with Dirichlet boundary conditions
\begin{align}\label{e31}\left\{\begin{array}{ll}
\displaystyle  \frac{\partial u(t, x)}{\partial t}
   =  \frac{\partial^{2}}{\partial x^{2}} u(t, x)-2 u(t, x)+\sin \left(\frac{1}{2+\cos t+\cos \sqrt{3} t}\right) u(t, x)  \\
 \quad  \quad  \quad \quad \,\,\, \,\,+\displaystyle\frac{1}{10}\left[\cos \left(\frac{1}{2+\sin t+\sin \sqrt{7} t}\right)+3e^{-|t|}+\frac{2}{1+t^2}\right] \sin u(t, x) \\
 \quad \quad  \quad  \,\,\,\,\, \quad+\frac{1}{5}[2+\sin u(t,x)]\arctan t, \,\,\quad t \in \mathbb{R},\,\,  x \in[0, \pi],  \\
  u(t, 0)= u(t, \pi)=0, \quad \,\, t \in \mathbb{R}.\end{array}\right.
\end{align}

Take
$\mathbb{X}=\mathcal{L}^{2}[0, \pi]$ with norm
$\|\cdot\|$ and inner product
$(\cdot, \cdot)$.
Define a linear operator
$A:D(A) \subset \mathbb{X} \rightarrow \mathbb{X}$
as
\begin{equation*}
  A u=\frac{\partial^{2} u}{\partial x^{2}}-2 u,
\end{equation*}
where
\begin{align*}
  D(A) =\big\{ u(\cdot) \in \mathbb{X}: u^{\prime \prime} \in \mathbb{X}, u^{\prime} \in \mathbb{X} \text { is absolutely continuous on } [0, \pi],\,\, u(0)=u(\pi)=0\big\}.
\end{align*}
According to \cite{ref26}, we know that $A$ is self-adjoint, with compact resolvent and is the infinitesimal generator of an analytic semigroup
$\{T(t)\}_{t \geq 0}$ on $\mathbb{X}$ satisfying
\begin{equation*}
  \|T(t)\| \leq e^{-3 t} \quad \text { for }\quad t>0.
\end{equation*}
In addition,
\begin{equation*}
  T(t) u=\sum\limits_{n=1}^{+\infty} e^{\left(-n^{2}+2\right) t}\left(u, v_{n}\right) v_{n},\,\, t \geq 0,\,\, u \in \mathbb{X},
\end{equation*}
where
$v_{n}(u)=\sqrt{\frac{2}{\pi}} \sin (n u)$.
Consider a family of linear operators
$A(t)$ defined by
\begin{equation*}
  A(t) u(x)=\left(A+\sin \left(\frac{1}{2+\cos t+\cos \sqrt{3} t}\right)\right) u(x), \,\,  x \in[0, \pi],\,\, u \in D(A(t)),
\end{equation*}
where $D(A)=D(A(t))$.
Then, the following Cauchy problem
\begin{align*}\left\{\begin{array}{ll}
  u^{\prime}(t)=A(t) u(t), \,\, t>s,\\
  u(s)=u \in \mathbb{X}\end{array}\right.
\end{align*}
has an associated evolution family
$\{\mathcal{U}(t, s)\}_{t \geq s}$  on
$\mathbb{X}$,
which can be explicitly given by
\begin{equation*}
  \mathcal{U}(t, s) u=\left(T(t-s) e^{\int_{s}^{t} \sin \left(\frac{1}{2+\cos \tau+\cos \sqrt{3} \tau}\right)d \tau} \right) u.
\end{equation*}
For every sequence $\{\gamma_{n}\}\subset \mathbb{R}$, we have
\begin{align*}
 &\|\mathcal{U}(t+\gamma_{n},s+\gamma_{n})-\mathcal{U}(t,s)\|\\
 \leq& e^{-(t-s)}\bigg|\sin \left(\frac{1}{2+\cos (\xi+\gamma_{n})+\cos \sqrt{3} (\xi+\gamma_{n})}\right)-\sin \left(\frac{1}{2+\cos \xi+\cos \sqrt{3} \xi}\right)\bigg|,
\end{align*}
where $\xi=\xi(s,t) \in [s,t]$.
Since $\sin \Big(\frac{1}{2+\cos t+\cos \sqrt{3} t}\Big)$ is almost automorphic, $\mathcal{U}(t,s)$ is bi-almost automorphic;
moreover, we have
\begin{equation*}
  \|\mathcal{U}(t, s)\| \leq e^{-2(t-s)} \quad \text { for all}\quad t \geq s.
\end{equation*}
Thus, $(H_1)$ and $(H_2)$ are verified. Also,
according to \cite{ex2}, we see that $A(t)$ satisfies conditions $(S_1)$ and $(S_2)$.

Let
\begin{equation*}
f(t, u(x))=f_1(t,u(x))+f_2(t,u(x)),
\end{equation*}
where
\begin{align*}
  f_1(t,u(x))=&\frac{1}{10}\left[\cos \left(\frac{1}{2+\sin t+\sin \sqrt{7} t}\right)+3e^{-|t|}+\frac{2}{1+t^2}\right] \sin u(x),\\
  f_2(t,u(x))=&\frac{1}{5}[2+\sin u(x)]\arctan t,
\end{align*}
then
\eqref{e31}
can be transformed  into the abstract equation \eqref{e1}.

Take the Radon-Nikodym derivatives of $\mu$ and $\nu$ are $\cos^2 t+2$ and $e^{|t|}+3+\cos t$, respectively. Then, we have
\begin{align*}
 \limsup\limits_{l\rightarrow\infty}\frac{\mu([-l,l])}{\nu([-l,l])}=\limsup\limits_{l\rightarrow\infty}\frac{\int_{-l}^{l}(\cos^2t+2)\mathrm{d}t}{\int_{-l}^{l}(e^{|t|}+3+\cos t)\mathrm{d}t}<\infty
\end{align*}
and
\begin{align*}
  \mu_\omega(A)=\int_{\omega+A}(\cos^2t+2)\mathrm{d}t\leq 3 \mu (A)
\end{align*}
for all $\omega\in \mathbb{R}$ and $A\in \Sigma$. Hence,
  conditions $(B_1)$ and $(B_2)$ are fulfilled.

 Since  $\|2e^{-|t|}\|_{W^p}=0$ and $\|\frac{2}{1+t^2}\|_{W^p}=0$, $\cos \Big(\frac{1}{2+\sin t+\sin \sqrt{7} t}\Big)$ is almost automorphic. Hence,    $f_1  \in W^p_{AA}(\mathbb{R}\times \mathbb{X},\mathbb{X})$.
Note that
\begin{align*}
& \lim\limits_{l\rightarrow +\infty}\frac{1}{\nu([-l,l])}\int_{-l}^{l} \bigg(\lim\limits_{h\rightarrow +\infty}\frac{1}{h}\int_{\theta}^{\theta+h}\bigg|\frac{1}{5}[2+\sin u(x)]\arctan s\bigg|^pds \bigg)^{\frac{1}{p}}d\mu(\theta)\\
 \leq& \frac{3\pi}{10}\lim\limits_{l\rightarrow\infty}\frac{\int_{-l}^{l}(\cos^2t+2)\mathrm{d}t}{\int_{-l}^{l}(e^{|t|}+3+\cos t)\mathrm{d}t}=0.
\end{align*}
Thus, $f_2\in \mathcal{O}^p(\mathbb{R}\times \mathbb{X},\mathbb{X},\mu,\nu)$. Therefore, condition $(H_5)$ and $(H_6)$ are verified  with $M_f=\frac{3}{5},K=1,$ $\lambda=2$ and $L_f=\frac{6+\pi}{10}$.

Consequently, by Theorem \ref{th51}, system \eqref{e31} has a unique $p$-th Weyl $(\mu,\nu)$-pseudo almost automorphic mild solution that is globally exponentially stable.

\section*{Conflict of interest statement}

The author declares that he has no conflicts of interest.

\section*{Data availability statement}

Data sharing is not applicable to this article as no datasets were generated or analyzed
during this study.

\end{document}